\documentclass[a4paper,12pt]{amsart}
\usepackage{amssymb}
\usepackage{ifthen}
\usepackage[usenames]{color}
\usepackage{graphicx}
\usepackage[active]{srcltx}
 \makeatletter
\renewcommand*\subjclass[2][2000]{%
  \def\@subjclass{#2}%
  \@ifundefined{subjclassname@#1}{%
    \ClassWarning{\@classname}{Unknown edition (#1) of Mathematics
      Subject Classification; using '1991'.}%
  }{%
    \@xp\let\@xp\subjclassname\csname subjclassname@#1\endcsname
  }%
}
 \makeatother
\usepackage{enumerate,url,amssymb,  mathrsfs}

\newtheorem{theorem}{Theorem}[section]

\newtheorem{corollary}{Corollary}[section]
\newtheorem{claim}{Claim}[section]

\newtheorem{lemma}{Lemma}[section]

\theoremstyle{definition}

\newtheorem{conjecture}{Conjecture}[section]

\theoremstyle{remark}
\newtheorem{remark}{Remark}[section]

\numberwithin{equation}{section}

\DeclareMathOperator{\sign}{sign}

\DeclareMathOperator*{\esssup}{ess\,sup}

\def\XXint#1#2#3{{\setbox0=\hbox{$#1{#2#3}{\int}$}
\vcenter{\hbox{$#2#3$}}\kern-.5\wd0}}

{}

\begin{document}

\title{Khavinson Problem for hyperbolic harmonic mappings in Hardy space}

\author{Jiaolong Chen}
\address{Jiaolong Chen, Key Laboratory of High Performance Computing and Stochastic Information Processing (HPCSIP)
(Ministry of Education of China), School of Mathematics and Statistics, Hunan Normal University, Changsha, Hunan 410081, People's Repulic of China}
\email{jiaolongchen@sina.com}

\author[David Kalaj]{David Kalaj}
\address{University of Montenegro, Faculty of Natural Sciences and
Mathematics, Cetinjski put b.b. 81000 Podgorica, Montenegro}
\email{davidkalaj@gmail.com}

\author[Petar Melentijevi\'{c}]{Petar Melentijevi\'{c}}
\address{Matemati\v{c}ki fakultet, University of Belgrade, Serbia}
\email{petarmel@matf.bg.ac.rs}



\maketitle

\makeatletter\def\thefootnote{\@arabic\c@footnote}\makeatother

\begin{abstract}
In this paper,
we partly solve the generalized Khavinson conjecture in the setting of hyperbolic harmonic mappings in Hardy space.
Assume that  $u=\mathcal{P}_{\Omega}[\phi]$ and
 $\phi\in L^{p}(\partial\Omega, \mathbb{R})$,
where $p\in[1,\infty]$,  $\mathcal{P}_{\Omega}[\phi]$ denotes the Poisson integral of $\phi$ with respect to the
hyperbolic Laplacian operator $\Delta_{h}$ in $\Omega$,
and $\Omega$ denotes the unit ball $\mathbb{B}^{n}$ or the half-space $\mathbb{H}^{n}$.
For any $x\in \Omega$ and $l\in \mathbb{S}^{n-1}$,
let $\mathbf{C}_{\Omega,q}(x)$ and $\mathbf{C}_{\Omega,q}(x;l)$ denote the optimal numbers for the gradient estimate
$$
 |\nabla u(x)|\leq \mathbf{C}_{\Omega,q}(x)\|\phi\|_{ L^{p}(\partial\Omega, \mathbb{R})}
$$
and  gradient estimate in the direction $l$
$$|\langle\nabla u(x),l\rangle|\leq \mathbf{C}_{\Omega,q}(x;l)\|\phi\|_{ L^{p}(\partial\Omega, \mathbb{R})},
$$
 respectively.
 Here $q$ is the conjugate of $p$.
 If $q=\infty$ or  $q\in[\frac{2K_{0}-1}{n-1}+1,\frac{2K_{0}}{n-1}+1]\cap [1,\infty)$ with $K_{0}\in\mathbb{N}=\{0,1,2,\ldots\}$,
then $\mathbf{C}_{\mathbb{B}^{n},q}(x)=\mathbf{C}_{\mathbb{B}^{n},q}(x;\pm\frac{x}{|x|})$ for any $x\in\mathbb{B}^{n}\backslash\{0\}$,
and $\mathbf{C}_{\mathbb{H}^{n},q}(x)=\mathbf{C}_{\mathbb{H}^{n},q}(x;\pm e_{n})$ for any $x\in \mathbb{H}^{n}$,
where $e_{n}=(0,\ldots,0,1)\in\mathbb{S}^{n-1}$.
However, if $q\in(1,\frac{n}{n-1})$,
then $\mathbf{C}_{\mathbb{B}^{n},q}(x)=\mathbf{C}_{\mathbb{B}^{n},q}(x;t_{x})$ for any $x\in\mathbb{B}^{n}\backslash\{0\}$,
and $\mathbf{C}_{\mathbb{H}^{n},q}(x)=\mathbf{C}_{\mathbb{H}^{n},q}(x;t_{e_{n}})$ for any $x\in \mathbb{H}^{n}$.
Here  $t_{w}$ denotes any unit vector in $\mathbb{R}^{n}$
 such that $\langle t_{w},w\rangle=0$ for $w\in \mathbb{R}^{n}\setminus\{0\}$.
\end{abstract}

\maketitle

\section{Introduction}\label{intsec}
 For $n\geq2$, let $\mathbb{R}^{n}$ denote the $n$-dimensional Euclidean space.
 For $x=(x_{1},\ldots,x_{n})\in \mathbb{R}^{n}$, sometimes we  identify each point $x$ with a column vector.
For two column vectors $x,y\in \mathbb{R}^{n}$, we use $\langle x,y\rangle$ to denote the inner product of $x$ and $y$.
The ball $\{x\in\mathbb{R}^{n}:|x|<r\}$ and the sphere $\{x\in\mathbb{R}^{n}:|x|=r\}$ are denoted by $\mathbb{B}^{n}(r)$ and $\mathbb{S}^{n-1}(r)$, respectively.
In particular, let $\mathbb{B}^{n}=\mathbb{B}^{n}(1)$, $\mathbb{S}^{n-1}=\mathbb{S}^{n-1}(1)$,
$\mathbb{R}^{2}=\mathbb{C}$
 and $\mathbb{B}^{2}=\mathbb{D}$.
 We also denote
$\mathbb{S}^{n-1}_{+}=\{x\in \mathbb{S}^{n-1}: x_n > 0\} $,
$\mathbb{S}^{n-1}_{-}=\{x\in \mathbb{S}^{n-1}: x_n <0\}$
and $\mathbb{H}^{n}=\{x=(x', x_{n})\in\mathbb{ R}^{n}: x' \in \mathbb{R}^{n-1},x_{n}>0\}$.
As is usual we identify $\mathbb{R}^{n-1}$ with $\mathbb{R}^{n-1}\times\{0\}$.
With this convention we have
that $\partial\mathbb{H}^{n} = \mathbb{R}^{n-1}$.

A mapping $u\in C^{2}(\mathbb{B}^{n}, \mathbb{R})$ is said to be {\it hyperbolic harmonic} if $u$ satisfies the hyperbolic Laplace equation
$$
\Delta_{h}u(x)=(1-|x|^2)^2\Delta u(x)+2(n-2)(1-|x|^2)\sum_{i=1}^{n} x_{i} \frac{\partial u}{\partial x_{i}}(x)=0,
$$
 where
 $\Delta$ denotes the usual Laplacian in $\mathbb{R}^{n}$.
 Meanwhile, a mapping $u\in C^{2}(\mathbb{H}^{n}, \mathbb{R})$ is said to be hyperbolic harmonic if $u$ satisfies the hyperbolic Laplace equation
$$
\Delta_{h}u(x)=x_{n}^2\Delta u(x)-(n-2)x_{n}\frac{\partial u}{\partial x_{n}}(x)=0.
$$
For convenience, in the rest of this paper, we call $\Delta_{h}$ the {\it hyperbolic Laplacian operator}.

When $n=2$, we easily see that hyperbolic harmonic mappings coincide with harmonic mappings.
In this paper, we focus our investigations on the case when $n\geq 3$.

\subsection{Hardy space for hyperbolic harmonic mappings}
For  $p\in(0,\infty]$,
 let $L^{p}(\mathbb{S}^{n-1},\mathbb{R})$ denote the space of Lebesgue measurable mappings from $\mathbb{S}^{n-1}$ into $\mathbb{R}$ satisfying $\|f\|_{L^{p}(\mathbb{S}^{n-1},\mathbb{R})}<\infty$, where
$$\|f\|_{L^{p}(\mathbb{S}^{n-1},\mathbb{R})}=\begin{cases}
\displaystyle \;\left( \int_{\mathbb{S}^{n-1}} |f(\xi)|^{p}d\sigma(\xi)\right)^{\frac{1}{p}} , & \text{ if } p\in (0,\infty),\\
\displaystyle \;\esssup_{\xi\in \mathbb{S}^{n-1}} \big\{|f(\xi)|\big\} , \;\;\;\;& \text{ if } p=\infty.
\end{cases}$$
Here and hereafter, $d \sigma$ denotes the normalized surface measure on $\mathbb{S}^{n-1}$ so that $\sigma(\mathbb{S}^{n-1})=1$.
Similarly,  let
 $L^{p}(\mathbb{R}^{n-1},\mathbb{R})$ denote the space of Lebesgue measurable mappings from $\mathbb{R}^{n-1}$ into $\mathbb{R}$ satisfying $\|f\|_{L^{p}(\mathbb{R}^{n-1},\mathbb{R})}<\infty$,
  where
$$\|f\|_{L^{p}(\mathbb{R}^{n-1},\mathbb{R})}=\begin{cases}
\displaystyle \;\left( \int_{\mathbb{R}^{n-1}} |f(x)|^{p}dV_{n-1}(x)\right)^{\frac{1}{p}}, & \text{ if } p\in (0,\infty),\\
\displaystyle \;\esssup_{x\in \mathbb{R}^{n-1}} \big\{|f(x)|\big\} , \;\;\;\;& \text{ if } p=\infty.
\end{cases}$$
Here and hereafter, $d V_{n-1}$ denotes the Lebesgue volume measure in $\mathbb{R}^{n-1}$.

If $\phi\in L^{1}(\mathbb{S}^{n-1},\mathbb{R})$, we define the {\it invariant Poisson integral} or {\it Poisson-Szeg\"{o} integral} of $\phi$ in $\mathbb{B}^{n}$ by (cf. \cite[Definition 5.3.2]{sto2016})
$$
\mathcal{P}_{\mathbb{B}^{n}}[\phi](x)=\int_{\mathbb{S}^{n-1}} \mathcal{P}_{\mathbb{B}^{n}}(x,\zeta)\phi(\zeta)d\sigma(\zeta),
$$
where
\begin{eqnarray}\label{eq-1.1}
\mathcal{P}_{\mathbb{B}^{n}}(x,\zeta)=\left(\frac{1-|x|^2}{|x-\zeta|^{2}}\right)^{n-1}
\end{eqnarray}
is the {\it Poisson-Szeg\"{o} kernel} with respective to $\Delta_{h}$ satisfying
$$
\int_{\mathbb{S}^{n-1}} \mathcal{P}_{\mathbb{B}^{n}}(x,\zeta) d\sigma(\zeta)=1
$$
(cf. \cite[Lemma 5.3.1(c)]{sto2016}).

If $\phi\in L^{1}(\mathbb{R}^{n-1},\mathbb{R})$, then the Poisson-Szeg\"{o} integral of $\phi$ is the function $\mathcal{P}_{\mathbb{H}^{n}}[\phi]$  defined by
(cf. \cite[Section 5.6]{sto2016})
$$
\mathcal{P}_{\mathbb{H}^{n}}[\phi](x)
=\int_{\mathbb{R}^{n-1}} \mathcal{P}_{\mathbb{H}^{n}}(x,y')\phi(y')dV_{n-1}(y'),
$$
where $x=(x',x_{n})\in \mathbb{H}^{n}$, $x',y'\in \mathbb{R}^{n-1}$ and
\begin{eqnarray}\label{eq-1.2}
\mathcal{P}_{\mathbb{H}^{n}}(x,y')=c_{n}\left(\frac{x_{n}}{|x'-y'|^{2}+x_{n}^{2}}\right)^{n-1}
\end{eqnarray}
is the Poisson-Szeg\"{o} kernel with respective to $\Delta_{h}$ satisfying
$$
\int_{\mathbb{R}^{n-1}} \mathcal{P}_{\mathbb{H}^{n}}(x,y') dV_{n-1}(y')=1.
$$
By calculations, we have
\begin{eqnarray}\label{eq-1.3}
c_{n}=\frac{2^{n-2} \Gamma(\frac{n}{2})}{ \pi^{\frac{n}{2}}},
\end{eqnarray}
where $\Gamma$ is the Gamma function.

 For $p\in[1,\infty]$, we use
 $\mathcal{H}^p(\mathbb{B}^{n},\mathbb{R})$ to denote the Hardy space
 consisting of hyperbolic harmonic mappings of the form $u=\mathcal{P}_{\mathbb{B}^{n}}[\phi]$
  with $\phi\in L^p(\mathbb{S}^{n-1},\mathbb{R})$.
Similarly, by $\mathcal{H}^p(\mathbb{H}^{n},\mathbb{R})$ we denote the Hardy space consisting of hyperbolic harmonic mappings of the form
$u=\mathcal{P}_{\mathbb{H}^{n}}[\phi]$
with $ \phi\in L^p(\mathbb{R}^{n-1},\mathbb{R}) $ and $p\in[1,\infty]$.

\subsection{The Khavinson problem}
Let $p\in[1,\infty]$ and $q$ be its conjugate.
Assume that $u=\mathcal{P}_{\Omega}[\phi]$ and $\phi\in L^{p}(\partial\Omega,\mathbb{R})$,
where $\Omega$ denotes $\mathbb{B}^{n}$ or $\mathbb{H}^{n}$.
For fixed $x\in \Omega$ and $l\in \mathbb{S}^{n-1}$,
let $\mathbf{C}_{\Omega,q}(x)$ and $\mathbf{C}_{\Omega,q}(x;l)$ denote the optimal numbers for the gradient estimate
\begin{eqnarray}\label{eq-1.4}
 |\nabla u(x)|\leq \mathbf{C}_{\Omega,q}(x)\|\phi\|_{L^{p}(\partial\Omega,\mathbb{R})}.
\end{eqnarray}
and  gradient estimate in the direction $l$
\begin{eqnarray}\label{eq-1.5}
 |\langle\nabla u(x),l\rangle|\leq \mathbf{C}_{\Omega,q}(x;l)\|\phi\|_{L^{p}(\partial\Omega,\mathbb{R})}.
\end{eqnarray}
 respectively.
 Since
 $$
 |\nabla u(x)|=\sup_{l\in\mathbb{S}^{n-1}} |\langle\nabla u(x),l\rangle|,
 $$
 we clearly have
\begin{eqnarray}\label{eq-1.6}
\mathbf{ C}_{\Omega,q}(x)=\sup_{l\in\mathbb{S}^{n-1}} \mathbf{C}_{\Omega,q}(x;l).
\end{eqnarray}

The generalized Khavinson conjecture states:
\begin{conjecture}\label{conj-1.1}
Let $q\in[1,\infty]$ and $e_{n}=(0,\ldots,0,1)\in \mathbb{S}^{n-1}$.
Then
\begin{enumerate}
  \item for any $x\in \mathbb{B}^{n}\backslash \{0\}$,  we have
$$
\mathbf{C}_{\mathbb{B}^{n},q}(x)=\mathbf{C}_{\mathbb{B}^{n},q}(x;\pm \frac{x}{|x|});
$$
  \item for any $x\in \mathbb{H}^{n}$,  we have
$$
\mathbf{C}_{\mathbb{H}^{n},q}(x)=\mathbf{C}_{\mathbb{H}^{n},q}(x;\pm e_{n}).
$$
\end{enumerate}

\end{conjecture}

This conjecture actually dates back to 1992.
Khavinson \cite{kh} obtained a sharp pointwise estimate for the
radial derivative of bounded harmonic functions in $\mathbb{B}^{3}$.
In a private conversation with Gresin and Maz'ya, he conjectured that the same estimate holds for the norm of the gradient of bounded harmonic functions.
Later, this conjecture was formulated by Kresin and Maz'ya in \cite{kr} for bounded harmonic functions in $\mathbb{B}^{n}$.
In the same paper, they obtained the sharp inequalities for the radial and tangential derivatives of such functions and solved the analogous problem for harmonic functions with the $L^{p}$ integrable boundary values for $p=1$ and $p=2$.
Also, the same authors in \cite{kr2} solved the half-space analog of this problem for $p=1$, $p=2$ and $p=\infty$.
Later, Kalaj and Markovi\'{c} \cite{kama} established an analogous result for harmonic functions from $\mathbb{D}$ into $\mathbb{C}$ with $L^{p}$ integrable boundary values, $p\geq1$.

For bounded harmonic functions in high dimension, Markovi\'{c} \cite{mark} considered the problem in a special situation
 when $x\in \mathbb{B}^{n}$ and $x$
is near the boundary.
Kalaj \cite{kalaj} proved the conjecture in $\mathbb{B}^{4}$, and
Melentijevi\'{c} \cite{2019mele} confirmed the conjecture in $\mathbb{B}^{3}$.
Very recently, Liu \cite{liu}  showed that the  conjecture is true in $\mathbb{B}^{n}$ with $n\geq3$.
See \cite{bur, kalajvuo, mac, protter} and references therein for further discussions on the gradient estimates for analytic and harmonic functions.

The main purpose of this paper is to consider the generalized Khavinson conjecture in the setting of hyperbolic harmonic mappings in Hardy space.
For the case when $q\in[1,\infty)$, we have the following  results.

\begin{theorem}\label{thm-1.1}
Let $q\in(1,\frac{n}{n-1})$.

(1) For any  $x\in\mathbb{B}^{n}\backslash \{0\}$ and $l\in\mathbb{S}^{n-1}$,
$$
\mathbf{C}_{\mathbb{B}^{n},q}\big(x;\pm\frac{x}{|x|}\big)
\leq\mathbf{C}_{\mathbb{B}^{n},q}\big(x;l\big)
\leq\mathbf{C}_{\mathbb{B}^{n},q}\big(x;t_{x}\big)
=\mathbf{C}_{\mathbb{B}^{n},q}(x),
$$
where $t_{x}$ is any unit vector in $\mathbb{R}^{n}$ such that
$\langle t_{x},\frac{x}{|x|}\rangle=0$.

(2) For  any $x\in\mathbb{H}^{n} $ and $l\in\mathbb{S}^{n-1}$,
$$
\mathbf{C}_{\mathbb{H}^{n},q}\big(x;\pm e_{n}\big)
\leq\mathbf{C}_{\mathbb{H}^{n},q}\big(x;l\big)
\leq\mathbf{C}_{\mathbb{H}^{n},q}\big(x;t_{e_{n}}\big)
=\mathbf{C}_{\mathbb{H}^{n},q}(x).
$$

\end{theorem}

\begin{theorem}\label{thm-1.2}
Let $q\in[\frac{2K_{0}-1}{n-1}+1,\frac{2K_{0}}{n-1}+1]\cap [1,\infty)$, where $K_{0}\in \mathbb{N}=\{0,1,2,\ldots\}$.

 (1) For any  $x\in\mathbb{B}^{n}\backslash \{0\}$ and $l\in\mathbb{S}^{n-1}$,
$$
\mathbf{C}_{\mathbb{B}^{n},q}\big(x;t_{x}\big)\leq
\mathbf{C}_{\mathbb{B}^{n},q}\big(x;l\big)
\leq\mathbf{C}_{\mathbb{B}^{n},q}\big(x;\pm\frac{x}{|x|}\big)
=\mathbf{C}_{\mathbb{B}^{n},q}(x).
$$

(2) For  any  $x\in\mathbb{H}^{n} $ and $l\in\mathbb{S}^{n-1}$,
$$
\mathbf{C}_{\mathbb{H}^{n},q}\big(x;t_{e_{n}}\big)\leq
\mathbf{C}_{\mathbb{H}^{n},q}\big(x;l\big)
\leq\mathbf{C}_{\mathbb{H}^{n},q}\big(x;\pm e_{n}\big)
=\mathbf{C}_{\mathbb{H}^{n},q}(x).
$$

\end{theorem}

In the following, we give two special cases of Theorem \ref{thm-1.2}.

\begin{theorem}\label{thm-1.3}
(1) For any $x\in\mathbb{B}^{n}$ and $l\in\mathbb{S}^{n-1}$,
$$
\mathbf{C}_{\mathbb{B}^{n}, \frac{n}{n-1}}(x)\equiv
\mathbf{C}_{\mathbb{B}^{n}, \frac{n}{n-1}}(x;l)
\equiv
\frac{ 2(n-1) }{(1-|x|^2)^{\frac{2n-1}{n}}}
\left(\frac{ \Gamma(\frac{n}{2})  \Gamma(\frac{2n-1}{2n-2}) }{ \sqrt{\pi} \Gamma(\frac{n^{2}}{2n-2}) }(1+|x|^2)\right)^{\frac{n-1}{n}}
$$
and
$$
\mathbf{C}_{\mathbb{B}^{n},1}(x)\equiv
\mathbf{C}_{\mathbb{B}^{n},1}(x;l)\equiv
\frac{ 2(n-1) \Gamma(\frac{n}{2}) }{\sqrt{\pi} \Gamma(\frac{n+1}{2})(1-|x|^2)}.
$$

(2) For any $x\in\mathbb{H}^{n}$ and $l\in\mathbb{S}^{n-1}$,
$$
\mathbf{C}_{\mathbb{H}^{n}, \frac{n}{n-1}}(x)\equiv
\mathbf{C}_{ \mathbb{H}^{n},\frac{n}{n-1}}(x;l)
\equiv
\frac{(n-1) \Gamma(\frac{n}{2})}{2^{n+1}\sqrt{\pi }x_{n}^{\frac{2n-1}{n}}}
\left(\frac{ \Gamma(\frac{2n-1}{2n-2})}{\Gamma(\frac{n^{2}}{2n-2})}\right)^{\frac{n-1}{n}}
$$
and
\begin{eqnarray*}
\mathbf{ C}_{\mathbb{H}^{n},1}(x;l)
\equiv
\mathbf{ C}_{\mathbb{H}^{n},1}(x)
\equiv
\frac{ 2\Gamma(\frac{n}{2})}{ \sqrt{\pi}\Gamma(\frac{n-1}{2}) x_{n}}.
\end{eqnarray*}

\end{theorem}

For $a\in\mathbb{R}$ and $k\in \mathbb{N}$,
let $(a)_{k}$ denote the {\it factorial function} with $(a)_{0}=1$
and $(a)_{k}=a(a+1)\ldots (a+k-1)$.
If $a$ is neither zero nor a negative integer, then (cf. \cite[Page 23]{rain})
\begin{eqnarray}\label{eq-1.7}
(a)_{k}=\frac{\Gamma(a+k)}{\Gamma(a)}.
\end{eqnarray}
For $r,s\in\mathbb{Z}^{+}=\{1,2,3,\ldots\}$ and $x\in \mathbb{R}$,
we define the {\it generalized hypergeometric series} by
\begin{eqnarray}\label{eq-1.8}
\;_{r}F_{s}(a_{1},a_{2},\ldots,a_{r};b_{1},\ldots,b_{s};x)
=\sum_{k=0}^{\infty}\frac{(a_{1})_{k}(a_{2})_{k}\ldots(a_{r})_{k}}{k!(b_{1})_{k}\ldots(b_{s})_{k}}x^{k},
\end{eqnarray}
where $a_{i},b_{j}\in \mathbb{R}$ ($1\leq i\leq r$, $1\leq j\leq s$) and $b_{j}$ is neither zero nor a negative integer (cf. \cite[Chapter IV]{erd} or \cite[Chapter 5]{rain}).
If $r=s+1$, then the series converges for $|x|<1$ and diverges for $|x|>1$.
If $r=s+1$ and $\sum_{i=1}^{s}b_{i}-\sum_{i=1}^{r}a_{i}>0$,
then the series is absolutely convergent on $|x|=1$.

Let $\Omega=\mathbb{B}^{n}$ or $\mathbb{H}^{n}$.
Using an explicit formula for $\mathbf{C}_{\Omega,q}\big(x;t_{x}\big)$ and $\mathbf{C}_{\Omega,q}\big(x;\frac{x}{|x|}\big)$,
we can reformulate Theorems \ref{thm-1.1} and \ref{thm-1.2} as follows, respectively.

\begin{theorem}\label{thm-1.4}
Let $p\in(n,\infty)$ and $q$ be its conjugate.

 (1) If $u=\mathcal{P}_{\mathbb{B}^{n}}[\phi]$ and $\phi\in L^{p}(\mathbb{S}^{n-1},\mathbb{R})$,
then for any $x\in\mathbb{B}^{n}$, we have the following sharp
inequality:
\[\begin{split}
|\nabla u(x)|
\leq\;&
 2(n-1)\left( \frac{\Gamma(\frac{n}{2})\Gamma(\frac{q+1}{2}) (1+|x|^{2})^{(n-1)(q-1)}}{ \sqrt{\pi}\Gamma(\frac{q+n}{2})(1-|x|^2)^{n(q-1)+1} }\right)^{\frac{1}{q}}
\|\phi\|_{L^{p}(\mathbb{S}^{n-1},\mathbb{R})}\\
&\times\left(\;_{2}F_{1}\left(\frac{(n-1)(1-q)}{2},\frac{1-(n-1)(q-1)}{2};\frac{q+n}{2};
\frac{4|x|^{2}}{(1+|x|^{2})^{2}}\right)\right)^{\frac{1}{q}}.
\end{split}\]

 (2) If $u=\mathcal{P}_{\mathbb{H}^{n}}[\phi]$ and $\phi\in L^{p}(\mathbb{R}^{n-1},\mathbb{R})$,
then for any $x\in\mathbb{H}^{n}$, we have the following sharp
inequality:
\[\begin{split}
|\nabla u(x)|
\leq\;&
\frac{(n-1) \Gamma(\frac{n}{2})}{ 2^{\frac{q-1}{q}}\pi^{\frac{nq-n+1}{2q}}x_{n}^{\frac{n(q-1)+1}{q}}}
\|\phi\|_{ L^{p}(\mathbb{R}^{n-1},\mathbb{R})}\\
&\times\left(\frac{\Gamma(\frac{q+1}{2}) }{ \Gamma(\frac{q+n}{2})}
\;_{3}F_{2}\left(\frac{(n-1)(1-q)}{2},\frac{1-(n-1)(q-1)}{2},\frac{q+1}{2};\frac{1}{2},\frac{q+n}{2};
1\right)\right)^{\frac{1}{q}}.
\end{split}\]

\end{theorem}

\begin{theorem}\label{thm-1.5}
Let $p\in\{\infty\} \cup[\frac{2K_{1}+n-1}{2K_{1}},\frac{2K_{1}+n-2}{2K_{1}-1}]$ and $q$ be its conjugate, where $K_{1}\in\mathbb{Z}^{+}$.

  (1) If $u=\mathcal{P}_{\mathbb{B}^{n}}[\phi]$ and $\phi\in L^{p}(\mathbb{S}^{n-1},\mathbb{R})$,
then for any $x\in\mathbb{B}^{n}$, we have the following sharp
inequality:
\[\begin{split}
|\nabla u(x)|
\leq&
 2(n-1)\left( \frac{\Gamma(\frac{n}{2})\Gamma(\frac{q+1}{2}) (1+|x|^{2})^{(n-1)(q-1)}}{ \sqrt{\pi}\Gamma(\frac{q+n}{2})(1-|x|^2)^{n(q-1)+1} }\right)^{\frac{1}{q}}
\|\phi\|_{L^{p}(\mathbb{S}^{n-1},\mathbb{R})}\\
&\times\left( \;_{3}F_{2}\left(\frac{(n-1)(1-q)}{2},\frac{1-(n-1)(q-1)}{2},\frac{q+1}{2};\frac{1}{2},\frac{q+n}{2};
\frac{4|x|^{2}}{(1+|x|^{2})^{2}}\right)\right)^{\frac{1}{q}}.
\end{split}\]

  (2) If $u=\mathcal{P}_{\mathbb{H}^{n}}[\phi]$ and $\phi\in L^{p}(\mathbb{R}^{n-1},\mathbb{R})$,
then for any $x\in\mathbb{H}^{n}$, we have the following sharp
inequality:
\[\begin{split}
|\nabla u(x)|
\leq&
\frac{(n-1) \Gamma(\frac{n}{2})}{ 2^{\frac{q-1}{q}}\pi^{\frac{nq-n+1}{2q}}x_{n}^{\frac{n(q-1)+1}{q}}}
\|\phi\|_{ L^{p}(\mathbb{R}^{n-1},\mathbb{R})}\\
&\times\left(\frac{\Gamma(\frac{q+1}{2}) }{ \Gamma(\frac{q+n}{2})}
\;_{2}F_{1}\left(\frac{(n-1)(1-q)}{2},\frac{1-(n-1)(q-1)}{2};\frac{q+n}{2};
1\right)\right)^{\frac{1}{q}}.
\end{split}\]

\end{theorem}

By using the expression of generalized hypergeometric series and  letting $q=1$, $\frac{n+1}{n-1}$ or $\frac{n}{n-1}$
in Theorem \ref{thm-1.5}, respectively, we easily arrive at the following three results, and we omit the proofs of them.
\begin{corollary}\label{cor-1.1}
(1)  If $u=\mathcal{P}_{\mathbb{B}^{n}}[\phi]$ and $\phi\in L^{\infty}(\mathbb{S}^{n-1},\mathbb{R})$,
then for any $x\in\mathbb{B}^{n}$, we have the following sharp
inequality:
\begin{eqnarray*}
|\nabla u(x)|\leq
\frac{2(n-1)\Gamma(\frac{n}{2})
\|\phi\|_{L^{\infty}(\mathbb{S}^{n-1},\mathbb{R})}}{\sqrt{\pi}(1-|x|^{2})\Gamma(\frac{n+1}{2})}.
\end{eqnarray*}

(2) If $u=\mathcal{P}_{\mathbb{H}^{n}}[\phi]$ and $\phi\in L^{\infty}(\mathbb{R}^{n-1},\mathbb{R})$,
then for any $x\in\mathbb{H}^{n}$, we have the following sharp
inequality:
\begin{eqnarray*}
|\nabla u(x)|\leq
\frac{ 2\Gamma(\frac{n}{2})\|\phi\|_{L^{\infty}(\mathbb{R}^{n-1},\mathbb{R})}}{ \sqrt{\pi}\Gamma(\frac{n-1}{2}) x_{n}}.
\end{eqnarray*}

\end{corollary}

\begin{corollary}\label{cor-1.2}
(1)  If $u=\mathcal{P}_{\mathbb{B}^{n}}[\phi]$ and  $\phi\in L^{\frac{n+1}{2}}(\mathbb{S}^{n-1},\mathbb{R})$,
then for any $x\in\mathbb{B}^{n}$, we have the following sharp
inequality:
$$
|\nabla u(x)|\leq
\frac{ 2(n-1)\|\phi\|_{L^{\frac{n+1}{2}}(\mathbb{S}^{n-1},\mathbb{R})} }{(1-|x|^2)^{\frac{3n-1}{n+1}}}
\left( \frac{ \Gamma(\frac{n}{2})\Gamma\left(\frac{n}{n-1}\right) }
 {\sqrt{\pi} \Gamma\left(\frac{n^{2}+1}{2n-2}\right)} (1+|x|^{2})^{2}
+\frac{4 \Gamma(\frac{n}{2})\Gamma\left(\frac{2n-1}{n-1}\right) }
{ \sqrt{\pi} \Gamma\left(\frac{n^{2}+2n-1}{2n-2}\right)} |x|^{2}\right)^{\frac{n-1}{n+1}}.
$$

(2) If $u=\mathcal{P}_{\mathbb{H}^{n}}[\phi]$
and  $\phi\in L^{\frac{n+1}{2}}(\mathbb{S}^{n-1},\mathbb{R})$,
then for any $x\in\mathbb{B}^{n}$, we have the following sharp
inequality:
$$
|\nabla u(x)|\leq
\frac{ (n-1)\Gamma(\frac{n}{2})\|\phi\|_{L^{\frac{n+1}{2}}(\mathbb{R}^{n-1},\mathbb{R})} }{2^{\frac{2}{n+1}} \pi^{\frac{3n-1}{2n-2}}x_{n}^{\frac{3n-1}{n+1}}}
\left( \frac{ \Gamma\left(\frac{n}{n-1}\right) }
 {  \Gamma\left(\frac{n^{2}+1}{2n-2}\right)}
+\frac{ \Gamma\left(\frac{n}{n-1}\right) }
{2 \Gamma\left(\frac{n^{2}+2n-1}{2n-2}\right)} \right)^{\frac{n-1}{n+1}}.
$$

\end{corollary}

\begin{corollary}\label{cor-1.3}
(1)  If $u=\mathcal{P}_{\mathbb{B}^{n}}[\phi]$ and  $\phi\in L^{n}(\mathbb{S}^{n-1},\mathbb{R})$,
then for any $x\in\mathbb{B}^{n}$, we have the following sharp
inequality:
\begin{eqnarray*}
|\nabla u(x)|\leq
\frac{ 2(n-1)\|\phi\|_{L^{n}(\mathbb{S}^{n-1},\mathbb{R})} }{(1-|x|^2)^{\frac{2n-1}{n}}}\left(\frac{ \Gamma(\frac{n}{2})  \Gamma(\frac{2n-1}{2n-2}) }{ \sqrt{\pi} \Gamma(\frac{n^{2}}{2n-2}) }(1+|x|^2)\right)^{\frac{n-1}{n}}.
\end{eqnarray*}

  (2) If $u=\mathcal{P}_{\mathbb{H}^{n}}[\phi]$ and  $\phi\in L^{n}(\mathbb{R}^{n-1},\mathbb{R})$,
then for any $x\in\mathbb{H}^{n}$, we have the following sharp
inequality:
\begin{eqnarray*}
|\nabla u(x)|\leq
 \frac{(n-1) \Gamma(\frac{n}{2})\|\phi\|_{L^{n}(\mathbb{R}^{n-1},\mathbb{R})}}
 {2^{n+1}\sqrt{\pi }x_{n}^{\frac{2n-1}{n}}}
\left(\frac{ \Gamma(\frac{2n-1}{2n-2})}{\Gamma(\frac{n^{2}}{2n-2})}\right)^{\frac{n-1}{n}}.
\end{eqnarray*}

\end{corollary}

For the case when $q=\infty$, we have the following result.

\begin{theorem}\label{thm-1.6}
(1)  For any $x\in\mathbb{B}^{n}\backslash \{0\}$ and $l\in\mathbb{S}^{n-1}$,
$$
\mathbf{C}_{\mathbb{B}^{n},\infty}\big(x;l\big)
\leq\mathbf{C}_{\mathbb{B}^{n},\infty}\big(x;\pm\frac{x}{|x|}\big)
=\mathbf{C}_{\mathbb{B}^{n},\infty}(x)
=2(n-1)
\frac{(1+|x|)^{n-2}}{(1-|x|)^{n}}.
$$

(2) For any $x\in\mathbb{H}^{n} $ and $l\in\mathbb{S}^{n-1}$,
$$
\mathbf{C}_{\mathbb{H}^{n},\infty}\big(x;l\big)
\leq\mathbf{C}_{\mathbb{H}^{n},\infty}\big(x;\pm e_{n}\big)
=\mathbf{C}_{\mathbb{H}^{n},\infty}(x)
 = \frac{2^{n-2} (n-1)\Gamma(\frac{n}{2})}{ \pi^{\frac{n}{2}}x_{n}^{n}}.
$$

\end{theorem}

Using an explicit formula for $\mathbf{C}_{\Omega,\infty}\big(x)$,
we can reformulate Theorem \ref{thm-1.6} as follows,
where $\Omega=\mathbb{B}^{n}$ or $\mathbb{H}^{n}$.

\begin{theorem}\label{thm-1.7}
(1)  If $u=\mathcal{P}_{\mathbb{B}^{n}}[\phi]$ and $\phi\in L^{1}(\mathbb{S}^{n-1},\mathbb{R})$,
then for any $x\in\mathbb{B}^{n}$, we have the following sharp
inequality:
\begin{eqnarray*}
|\nabla u(x)|\leq 2(n-1)
\frac{(1+|x|)^{n-2}}{(1-|x|)^{n}}\|\phi\|_{L^{1}(\mathbb{S}^{n-1},\mathbb{R})}.
\end{eqnarray*}

  (2) If $u=\mathcal{P}_{\mathbb{H}^{n}}[\phi]$ and $\phi\in L^{1}(\mathbb{R}^{n-1},\mathbb{R})$,
then for any $x\in\mathbb{H}^{n}$, we have the following sharp
inequality:
\begin{eqnarray*}
|\nabla u(x)|\leq
\frac{2^{n-2} (n-1)\Gamma(\frac{n}{2})}{ \pi^{\frac{n}{2}}x_{n}^{n}}\|\phi\|_{L^{1}(\mathbb{R}^{n-1},\mathbb{R})} .
\end{eqnarray*}

\end{theorem}

The rest of this paper is organized as follows.
In Section \ref{sec-2}, we establish some representations for $\mathbf{C}_{\mathbb{B}^{n},q}(x;l)$ when $q\in[1,\infty)$.
In Section \ref{sec-3}, we prove Theorems \ref{thm-1.1}$\sim$\ref{thm-1.5} for the unit ball case.
In Section \ref{sec-4},  we show Theorems \ref{thm-1.6} and \ref{thm-1.7}
for the unit ball case.
In Section \ref{sec-5}, we establish some representation for $\mathbf{C}_{\mathbb{H}^{n},q}(x;l)$ with $q\in[1,\infty)$.
In Section \ref{sec-6},  we present the proofs of Theorems \ref{thm-1.1}$\sim$\ref{thm-1.7} for the half-space case.

\section{Case $q\in[1,\infty)$ and representations for $\mathbf{C}_{\mathbb{B}^{n},q}(x;l)$}\label{sec-2}

In this section, we first establish a general integral representation formula for the sharp quantity $\mathbf{C}_{\mathbb{B}^{n},q}(x;l)$ when $q\in[1,\infty)$.

\begin{lemma}\label{lem-2.1}
For any $q\in[1,\infty)$, $x\in\mathbb{B}^{n}$ and $l\in \mathbb{S}^{n-1}$,
we have
\begin{eqnarray}\label{eq-2.1}
\quad\mathbf{C}_{\mathbb{B}^{n},q}(x;l)=\frac{ 2(n-1)}{(1-|x|^2)^{\frac{n(q-1)+1}{q}}} \left(\int_{\mathbb{S}^{n-1}}  |\eta-x|^{2(n-1)(q-1)} |\langle\eta,l\rangle|^{q}   d\sigma(\eta)\right)^{\frac{1}{q}}.
\end{eqnarray}
\end{lemma}

\begin{proof}
 In order to calculate $\mathbf{C}_{\mathbb{B}^{n},q}(x;l)$,
 we let $p$ be the conjugate of $q$ with $p\in(1,\infty]$, $ \phi \in L^{p}(\mathbb{S}^{n-1},\mathbb{R})$
 and $u=\mathcal{P}_{\mathbb{B}^{n}}[\phi]$ in $\mathbb{B}^{n}$.
For any $x\in\mathbb{B}^{n}$ and $\zeta\in\mathbb{S}^{n-1}$,
it follows from \eqref{eq-1.1} that
\begin{eqnarray}\label{eq-2.2}
\nabla \mathcal{P}_{\mathbb{B}^{n}}(x,\zeta)
=-2(n-1)(1-|x|^2)^{n-2}\;
\frac{x|x-\zeta|^{2}+(1-|x|^{2})(x-\zeta)}{|x-\zeta|^{2n}}
\end{eqnarray}
(cf. \cite[Equation (2.25)]{chen2018}).
Obviously, the mapping $(x,\zeta)\mapsto \nabla \mathcal{P}_{\mathbb{B}^{n}}(x,\zeta)$ is continuous in $\mathbb{B}^{n}(r)\times\mathbb{S}^{n-1}$, where $r\in(0,1)$.
Then one can easily obtain
\begin{eqnarray}\label{eq-2.3}
\langle\nabla u(x),l\rangle
=\int_{\mathbb{S}^{n-1}} \langle\nabla \mathcal{P}_{\mathbb{B}^{n}}(x,\zeta),l\rangle \phi(\zeta) d\sigma(\zeta),
\end{eqnarray}
and so,
\begin{eqnarray}\label{eq-2.4}
 |\langle\nabla u(x),l\rangle|\leq\left( \int_{\mathbb{S}^{n-1}} |\langle\nabla \mathcal{P}_{\mathbb{B}^{n}}(x,\zeta),l\rangle|^{q} d\sigma(\zeta)\right)^{\frac{1}{q}}\|\phi\|_{L^{p}(\mathbb{S}^{n-1},\mathbb{R})}.
\end{eqnarray}

 On the other hand, for every $p\in(1,\infty]$,
$x\in \mathbb{B}^{n}$ and $l\in \mathbb{S}^{n-1}$,
let
$$
\phi_{l}(\zeta)
= |\langle\nabla \mathcal{P}_{\mathbb{B}^{n}}(x,\zeta),l\rangle|^{q/p}
\sign\langle\nabla \mathcal{P}_{\mathbb{B}^{n}}(x,\zeta),l\rangle
$$
on $\mathbb{S}^{n-1}$ and $u_{l}=\mathcal{P}_{\mathbb{B}^{n}}[\phi_{l}]$ in $\mathbb{B}^{n}$.
 Then we have
$$
 |\langle\nabla u_{l}(x),l\rangle|= \left(\int_{\mathbb{S}^{n-1}} |\langle\nabla \mathcal{P}_{\mathbb{B}^{n}}(x,\zeta),l\rangle|^{q} d\sigma(\zeta)\right)^{\frac{1}{q}}\|\phi_{l}\|_{L^{p}(\mathbb{S}^{n-1},\mathbb{R})}.
$$
This, together with \eqref{eq-1.5} and \eqref{eq-2.4},
 implies that for any $q\in[1,\infty)$, $x\in\mathbb{B}^{n}$ and $l\in \mathbb{S}^{n-1}$,
\begin{eqnarray}\label{eq-2.5}
\mathbf{C}_{\mathbb{B}^{n},q}(x;l)
=\left(\int_{\mathbb{S}^{n-1}} |\langle\nabla \mathcal{P}_{\mathbb{B}^{n}}(x,\zeta),l\rangle|^{q} d\sigma(\zeta)\right)^{\frac{1}{q}}.
\end{eqnarray}

In the following, we calculate the integral above.
 For any $\eta\in \mathbb{S}^{n-1}$ and $x\in \mathbb{B}^{n}$,
let $\zeta=T_{x}(\eta)$, where
$$T_{x}(\eta)=x-(1-|x|^{2})\frac{\eta-x }{|\eta-x|^{2}}.$$
Then $\zeta=T_{x}(\eta)$ is a transformation from
  $\mathbb{S}^{n-1}$ onto $\mathbb{S}^{n-1}$
  (cf. \cite[Section 2.5]{chen2018}) satisfying
\begin{eqnarray}\label{eq-2.6}
x-\zeta=(1-|x|^{2})\frac{\eta-x}{|\eta-x|^{2}},\quad
 |x-\zeta|=\frac{1-|x|^{2}}{|\eta-x|}
 \end{eqnarray}
 and
\begin{eqnarray}\label{eq-2.7}
d\sigma(\zeta)=\frac{(1-|x|^2)^{n-1}}{|\eta-x|^{2n-2}}d\sigma(\eta)
 \end{eqnarray}
 (cf. \cite[Page 250]{mark}).
Combining  \eqref{eq-2.2} and \eqref{eq-2.6}, we get
 $$
\nabla \mathcal{P}_{\mathbb{B}^{n}}(x,\zeta)
=-2(n-1)\eta \;\frac{|\eta-x|^{2n-2}}{(1-|x|^2)^{n}},
$$
which, together with \eqref{eq-2.7},
yields that
 $$
|\langle\nabla \mathcal{P}_{\mathbb{B}^{n}}(x,\zeta),l\rangle|^{q} d\sigma(\zeta)
= 2^{q}(n-1)^{q}\frac{|\eta-x|^{2(n-1)(q-1)}\; |\langle\eta,l\rangle|^{q} }{(1-|x|^2)^{n(q-1)+1}} d\sigma(\eta).
$$
Therefore, for any  $q\in[1,\infty)$, $x\in\mathbb{B}^{n}$ and $l\in \mathbb{S}^{n-1}$,
$$
\int_{\mathbb{S}^{n-1}} |\langle\nabla\mathcal{P}_{\mathbb{B}^{n}}(x,\zeta),l\rangle|^{q} d\sigma(\zeta)
=\frac{ 2^{q}(n-1)^{q}  }{(1-|x|^2)^{n(q-1)+1}} \int_{\mathbb{S}^{n-1}}  |\eta-x|^{2(n-1)(q-1)}\; |\langle\eta,l\rangle|^{q}   d\sigma(\eta).
$$
 From this and \eqref{eq-2.5}, we see that Lemma \ref{lem-2.1} is true.
\end{proof}

Based on Lemma \ref{lem-2.1}, we obtain the following two results.
The first one is about the symmetry property of the quantity $\mathbf{C}_{\mathbb{B}^{n},q}(x;l)$, which is useful.

\begin{lemma}\label{lem-2.2}
For any $q\in[1,\infty)$, $x\in\mathbb{B}^{n}$, $l\in \mathbb{S}^{n-1}$ and unitary transformation $A$ in $\mathbb{R}^{n}$,
we have
\begin{eqnarray*}
\mathbf{C}_{\mathbb{B}^{n},q}(x;l)=\mathbf{C}_{\mathbb{B}^{n},q}(Ax;Al).
\end{eqnarray*}
\end{lemma}

\begin{proof}
For any unitary transformation $A$, by replacing $\eta$ with $A\eta$ in \eqref{eq-2.1}, we get
\[\begin{split}
&\mathbf{C}_{\mathbb{B}^{n},q}(Ax;Al)\\
=&\frac{ 2(n-1)}{(1-|x|^2)^{\frac{n(q-1)+1}{q}}} \left(\int_{\mathbb{S}^{n-1}}  |A\eta-Ax|^{2(n-1)(q-1)}\; |\langle A\eta,Al\rangle|^{q}   d\sigma(A\eta)\right)^{\frac{1}{q}}\\
=&\frac{ 2(n-1)}{(1-|x|^2)^{\frac{n(q-1)+1}{q}}} \left(\int_{\mathbb{S}^{n-1}}  |\eta-x|^{2(n-1)(q-1)}\; |\langle
\eta,l\rangle|^{q}   d\sigma(\eta)\right)^{\frac{1}{q}}
=\mathbf{C}_{\mathbb{B}^{n},q}(x;l),
\end{split}\]
as required.
\end{proof}

For any $q\in[1,\infty)$, $x\in \mathbb{B}^{n}$ and $l\in \mathbb{S}^{n-1}$,
let
\begin{eqnarray}\label{eq-2.8}
C_{\mathbb{B}^{n},q}(x; l)=\int_{\mathbb{S}^{n-1}}  |\eta-x|^{2(n-1)(q-1)}\; |\langle\eta,l\rangle|^{q}   d\sigma(\eta).
\end{eqnarray}
Then we obtain the following result.

\begin{lemma}\label{lem-2.3}
For any $q\in[1,\infty)$, $\alpha \in[0,\pi]$ and $\rho\in[0,1)$,
we have
\begin{eqnarray*}
&\;\;&C_{\mathbb{B}^{n},q}(\rho l_{\alpha}; e_{n})\\
&=&\frac{\Gamma(\frac{n}{2}) (1+\rho^{2})^{(n-1)(q-1)}}{\sqrt{\pi} }
\sum_{k=0}^{\infty}
\sum_{j=0}^{k}
\frac{(\frac{(n-1)(1-q)}{2})_{k}(\frac{1-(n-1)(q-1)}{2})_{k}}{\Gamma(k+\frac{q+n}{2})}
\left(\frac{2\rho}{1+\rho^{2}}\right)^{2k}
\\
&\;\;&
\times
\frac{ 4^{j}\Gamma(j+\frac{q+1}{2} )}
{(1)_{k-j}(2j)!}
\sin^{2k-2j}\alpha\cos^{2j}\alpha,
\end{eqnarray*}
where $l_{\alpha}=\sin\alpha \cdot e_{n-1}+\cos\alpha \cdot e_{n}$ and $e_{n-1}=(0,\ldots,0,1,0)\in\mathbb{S}^{n-1}$.
\end{lemma}

\begin{proof}
For any $q\in[1,\infty)$, $\alpha \in[0,\pi]$ and $\rho\in[0,1)$,
we deduce from \eqref{eq-2.8} and \cite[Lemma 1]{2019mele} that
\begin{eqnarray}\label{eq-2.9}
&\;\;&\frac{2\pi}{n-2} C_{\mathbb{B}^{n},q}(\rho l_{\alpha}; e_{n})\\\nonumber
&=&\frac{2\pi}{n-2}\int_{\mathbb{S}^{n-1}} |\eta_{n}|^{q} (1+\rho^{2}-2\rho\eta_{n} \cos\alpha-
2\rho\eta_{n-1} \sin \alpha)^{(n-1)(q-1)}d\sigma(\eta)\\\nonumber
&=&\int_{\mathbb{D}} |x|^{q}(1-x^2-y^2)^{\frac{n}{2}-2} (1+\rho^{2}-2\rho x \cos\alpha-
2\rho y \sin \alpha)^{(n-1)(q-1)}  dxdy\\\nonumber
&=&\int_{-1}^{1}|x|^{q}\int_{-\sqrt{1-x^{2}}}^{\sqrt{1-x^{2}}}(1-x^2-y^2)^{\frac{n}{2}-2} (1+\rho^{2}-2\rho x \cos\alpha-
2\rho y \sin \alpha)^{(n-1)(q-1)}  dydx,
\end{eqnarray}
where $\eta=(\eta_{1},\ldots,\eta_{n})\in \mathbb{S}^{n-1}$.
Let $y=t\sqrt{1-x^2}$, where $t\in(-1,1)$ and $x\in(-1,1)$.
Elementary calculations show that
\begin{eqnarray}\label{eq-2.10}
&\;\;&\int_{-\sqrt{1-x^{2}}}^{\sqrt{1-x^{2}}}(1-x^2-y^2)^{\frac{n}{2}-2} (1+\rho^{2}-2\rho x \cos\alpha-
2\rho y \sin \alpha)^{(n-1)(q-1)}  dy\\\nonumber
&=&(1-x^2)^{\frac{n-3}{2}}  \int_{-1}^{1} (1-t^2)^{\frac{n}{2}-2}(1+\rho^{2}-2\rho x \cos\alpha-
2\rho t\sqrt{1-x^2} \sin \alpha)^{(n-1)(q-1)} dt .
\end{eqnarray}
Since for any $m\geq0$ and $s\in(-1,1)$,
we have
$$(1-s)^{m}
 =\sum_{k=0}^{\infty} \Big(
\begin{array}{c}
m\\
k\\
\end{array}
\Big)(-s)^{k}
=\sum_{k=0}^{\infty} \frac{ (-m)_{k}}{k!}
s^{k}$$ (cf. \cite[Page 85]{erd}).
Then
\begin{eqnarray*}
&&(1+\rho^{2}-2\rho x \cos\alpha-
2\rho t\sqrt{1-x^2} \sin \alpha)^{(n-1)(q-1)}\\
&=& (1+\rho^{2})^{(n-1)(q-1)}\left(1-\frac{2\rho}{1+\rho^{2}} (x \cos\alpha+t\sqrt{1-x^2} \sin \alpha)\right)^{(n-1)(q-1)}\\
&=&(1+\rho^{2})^{(n-1)(q-1)}\sum_{k=0}^{\infty}\frac{\big((n-1)(1-q)\big)_{k}}{k!}\left( \frac{2\rho}{1+\rho^{2}} \right)^{k} (x \cos\alpha+t\sqrt{1-x^2} \sin \alpha)^{k}\\
&=&(1+\rho^{2})^{(n-1)(q-1)}\sum_{k=0}^{\infty}\frac{\big((n-1)(1-q)\big)_{k}}{k!}\left( \frac{2\rho}{1+\rho^{2}} \right)^{k} \\
&&\times
\sum_{j=0}^{\infty}
\frac{( -1)^{j}(-k)_{j}}{j!}x^{j} t^{k-j} (1-x^2)^{\frac{k-j}{2}} \sin^{k-j} \alpha\cos^{j}\alpha.
\end{eqnarray*}
This, together with \eqref{eq-2.9} and \eqref{eq-2.10}, shows that
\[\begin{split}
&\frac{2\pi}{n-2}C_{\mathbb{B}^{n},q}(\rho l_{\alpha}; e_{n})\\
 =&(1+\rho^{2})^{(n-1)(q-1)}
\sum_{k=0}^{\infty}
\sum_{j=0}^{\infty}
\frac{\big((n-1)(1-q)\big)_{k} }{k!}
\left(\frac{2\rho}{1+\rho^{2}}\right)^{k}
\frac{(-1)^{j}(-k)_{j}}{j!}
\\
&\times
\int_{-1}^{1} |x|^{q} x^{j}(1-x^2)^{\frac{n-3+k-j}{2}}dx
\int_{-1}^{1} t^{k-j} (1-t^2)^{\frac{n}{2}-2}dt
\cdot\sin^{k-j}\alpha\cos^{j}\alpha.
\end{split}\]
It follows from the fact (cf. \cite[Pages 18 and 19]{rain})
$$
\int_{-1}^{1} |x|^{q} x^{2j}(1-x^2)^{\frac{n-3}{2}+k-j}dx
=\int_{0}^{1} x^{ \frac{q-1}{2}+j} (1-x)^{\frac{n-3}{2}+k-j}dx
=\frac{\Gamma(\frac{q+1}{2}+j)\Gamma(\frac{n-1}{2}+k-j)}{\Gamma(k+\frac{n+q}{2})}
$$
and
$$
\int_{-1}^{1} t^{2k-2j} (1-t^2)^{\frac{n}{2}-2}dt
=\int_{0}^{1} t^{k-j-\frac{1}{2}} (1-t)^{\frac{n}{2}-2}dt
=\frac{\Gamma(k-j+\frac{1}{2})\Gamma(\frac{n}{2}-1)}{\Gamma(k-j+\frac{n-1}{2})}
$$
that
\[\begin{split}
 \frac{ \pi}{\Gamma(\frac{n}{2})}C_{\mathbb{B}^{n},q}(\rho l_{\alpha}; e_{n})
 =&(1+\rho^{2})^{(n-1)(q-1)}
\sum_{k=0}^{\infty}
\sum_{j=0}^{\infty}
\frac{\big((n-1)(1-q)\big)_{2k} }{\Gamma(2k+1)}
\left(\frac{2\rho}{1+\rho^{2}}\right)^{2k}
\frac{ (-2k)_{2j}}{\Gamma(2j+1)}
\\
&\times
\frac{\Gamma(\frac{q+1}{2}+j)\Gamma(k-j+\frac{1}{2}) }{\Gamma(k+\frac{n+q}{2})}
 \cdot\sin^{2k-2j}\alpha\cos^{2j}\alpha.
\end{split}\]
Further, for any $a\in \mathbb{R}$, $b>0$ and $k\in \mathbb{N}$,
we have (cf. \cite[Pages 23 and 24]{rain}).
\begin{eqnarray}\label{eq-2.11}
(2a)_{2k}=2^{2k}(a)_{k} \big(a+\frac{1}{2} \big)_{k}
\quad\text{and}\quad
 \sqrt{\pi}\Gamma(2b)=2^{2b-1}\Gamma(b)\Gamma(b+\frac{1}{2}).
 \end{eqnarray}
 Therefore,
\[\begin{split}
C_{\mathbb{B}^{n},q}(\rho l_{\alpha}; e_{n})
 =&  \Gamma(\frac{n}{2}) (1+\rho^{2})^{(n-1)(q-1)}
\sum_{k=0}^{\infty}
\sum_{j=0}^{k}
\frac{\big(\frac{(n-1)(1-q)}{2}\big)_{k}\big(\frac{1-(n-1)(q-1)}{2}\big)_{k} }{\Gamma(k+\frac{1}{2})\Gamma(j+\frac{1}{2})}
\left(\frac{2\rho}{1+\rho^{2}}\right)^{2k}
\\
&\times\frac{ (-k)_{j}(\frac{1}{2}-k)_{j}\Gamma(k-j+\frac{1}{2}) }{\Gamma(j+1)\Gamma(k+1) }
\frac{\Gamma(\frac{q+1}{2}+j) }{\Gamma(k+\frac{n+q}{2})}
 \sin^{2k-2j}\alpha\cos^{2j}\alpha.
\end{split}\]
Note that
 $$
 (-k)_{j}\Gamma(k-j+1)=(-1)^{j}\Gamma(k+1)
 \;\;\text{and}\;\;
  (\frac{1}{2}-k)_{j}\Gamma(k-j+\frac{1}{2})=(-1)^{j}\Gamma(k+\frac{1}{2}).
 $$
This, together with \eqref{eq-2.11}, means
\begin{eqnarray}\label{eq-2.12}
&\;\;&C_{\mathbb{B}^{n},q}(\rho l_{\alpha}; e_{n})\\\nonumber
 &=&  \Gamma(\frac{n}{2}) (1+\rho^{2})^{(n-1)(q-1)}
\sum_{k=0}^{\infty}
\sum_{j=0}^{k}
\frac{\big(\frac{(n-1)(1-q)}{2}\big)_{k}\big(\frac{1-(n-1)(q-1)}{2}\big)_{k} }{ \Gamma(j+\frac{1}{2})\Gamma(j+1)\Gamma(k-j+1)}
\\\nonumber
&&\times
\frac{\Gamma(\frac{q+1}{2}+j) }{\Gamma(k+\frac{n+q}{2})}
 \left(\frac{2\rho}{1+\rho^{2}}\right)^{2k} \sin^{2k-2j}\alpha\cos^{2j}\alpha\\\nonumber
&=&\frac{\Gamma(\frac{n}{2}) (1+\rho^{2})^{(n-1)(q-1)}}{\sqrt{\pi} }
\sum_{k=0}^{\infty}
\sum_{j=0}^{k}
\frac{(\frac{(n-1)(1-q)}{2})_{k}(\frac{1-(n-1)(q-1)}{2})_{k}}{\Gamma(k+\frac{q+n}{2})}
\left(\frac{2\rho}{1+\rho^{2}}\right)^{2k}
\\\nonumber
&\;\;&
\times
\frac{ 4^{j}\Gamma(j+\frac{q+1}{2} )}
{(1)_{k-j}(2j)!}
\sin^{2k-2j}\alpha\cos^{2j}\alpha.
\end{eqnarray}

The proof of the lemma is complete.
\end{proof}

\section{Proofs of Theorems \ref{thm-1.1}$\sim$\ref{thm-1.5} for the unit ball case}\label{sec-3}
The aim of this section is to prove Theorems \ref{thm-1.1}$\sim$\ref{thm-1.5} for the unit ball case.
First, we present the proofs of Theorems \ref{thm-1.2}(1) and  \ref{thm-1.5}(1),
which is based upon the ideas from \cite{2019mele2}.
Based on the proof of Theorem \ref{thm-1.2}(1), we prove the other results.

 \subsection{Proofs of Theorems \ref{thm-1.2}(1) and \ref{thm-1.5}(1)}\label{sec-3.1}
(I) First, we prove Theorem \ref{thm-1.2}(1).
For any $x\in \mathbb{B}^{n}$ and $l\in \mathbb{S}^{n-1}$,
in order to find the sharp quantity $\mathbf{C}_{\mathbb{B}^{n},q}(x;l)$,
 we choose an unity transformation $A$ in $\mathbb{R}^{n}$
 such that $Ax=\rho l_{\alpha}$ and $Al=e_{n}$,
where $\rho=|x|\in[0,1)$, $l_{\alpha}=\sin\alpha \cdot e_{n-1}+\cos\alpha \cdot e_{n}$ and $\alpha\in[0,\pi]$.
 By Lemma \ref{lem-2.2},  \eqref{eq-2.1} and \eqref{eq-2.8}, we see that
\begin{eqnarray}\label{eq-3.1}
\mathbf{C}_{\mathbb{B}^{n},q}(x;l)
&=&\mathbf{C}_{\mathbb{B}^{n},q}(Ax;Al)
=\mathbf{C}_{\mathbb{B}^{n},q}(\rho l_{\alpha}; e_{n})\\\nonumber
&=&\frac{ 2(n-1)}{(1-|x|^2)^{\frac{n(q-1)+1}{q}}}  C^{\frac{1}{q}}_{\mathbb{B}^{n},q}(\rho l_{\alpha}; e_{n}) .
\end{eqnarray}
Hence, to prove the theorem, it remains to calculate $C_{\mathbb{B}^{n},q}(\rho l_{\alpha}; e_{n})$.
By Lemma \ref{lem-2.3}, we obtain
\begin{eqnarray*}
&\;\;&C_{\mathbb{B}^{n},q}(\rho l_{\alpha}; e_{n})\\\nonumber
&=&\frac{\Gamma(\frac{n}{2}) (1+\rho^{2})^{(n-1)(q-1)}}{\sqrt{\pi} }
\sum_{k=0}^{\infty}
\sum_{j=0}^{k}
\frac{(\frac{(n-1)(1-q)}{2})_{k}(\frac{1-(n-1)(q-1)}{2})_{k}}{\Gamma(k+\frac{q+n}{2})}
\left(\frac{2\rho}{1+\rho^{2}}\right)^{2k}
\\\nonumber
&\;\;&
\times
\frac{ 4^{j}\Gamma(j+\frac{q+1}{2} )}
{(1)_{k-j}(2j)!}
\sin^{2k-2j}\alpha\cos^{2j}\alpha.
\end{eqnarray*}

In order to estimate the right side of the above equation,
for any $k\in \mathbb{N}$ and $j\in\{0,1,\ldots,k\}$,
we let
$$
B_{q,k}(j)=\frac{ 4^{j}\Gamma(j+\frac{q+1}{2} )}
{(1)_{k-j}(2j)!}.
$$
\begin{claim}\label{claim-3.2}
For any $k\in \mathbb{N}$ and $j\in\{0,1,\ldots,k\}$,
$$
 \left(
\begin{array}{c}
k\\
j\\
\end{array}
\right)\frac{\Gamma( \frac{q+1}{2} )}
{k!}
\leq B_{q,k}(j)\leq
4^{k}\left(
\begin{array}{c}
k\\
j\\
\end{array}
\right)\frac{\Gamma(k+\frac{q+1}{2} )}
{(2k)!}.
$$
The first equality occurs if and only if $j=0$, and
the second one occurs if and only if $j=k$.
\end{claim}

From the fact that
\begin{eqnarray*}
\left(
\begin{array}{c}
k\\
j\\
\end{array}
\right)
=\frac{\Gamma(k+1)}{\Gamma(j+1)\Gamma(k-j+1)},
\end{eqnarray*}
we obtain
$$
 B_{n,k}(j) \left/ \left(
\begin{array}{c}
k\\
j\\
\end{array}
\right)
\right.
=\frac{ 4^{j}\Gamma(j+\frac{q+1}{2} )\Gamma(j+1)}
{(2j)!\Gamma(k+1)}.
$$

For $k\in \mathbb{N}$ and $j\in\{0,1,\ldots,k\}$, set
$$
a_{q,k}(j)
=\frac{ 4^{j}\Gamma(j+\frac{q+1}{2} )\Gamma(j+1)}
{(2j)!\Gamma(k+1)}.
$$
Since
$$
\frac{a_{q,k}(j+1)}{a_{q,k}(j)}
=\frac{ 4(j+\frac{q+1}{2} )(j+1)}
{(2j+2)(2j+1)}
=\frac{2j+q+1}
{2j+1}>1,
$$
we get
\begin{eqnarray*}
\frac{ \Gamma(\frac{q+1}{2} )}
{k!}
=a_{q,k}(0)\leq B_{n,k}(j) \left/ \left(
\begin{array}{c}
k\\
j\\
\end{array}
\right)
\right.
=a_{q,k}(j)
\leq
a_{q,k}(k)
=\frac{ 4^{k}\Gamma(k+\frac{q+1}{2}) }
{(2k)!},
\end{eqnarray*}
which implies that Claim \ref{claim-3.2} is true.

\smallskip

By the assumption in the theorem that  $q\in[\frac{2K_{0}-1}{n-1}+1,\frac{2K_{0}}{n-1}+1]\cap [1,\infty)$ and $K_{0}\in \mathbb{N}$,
we see that for any $k\geq0$,
$$
\left(\frac{(n-1)(1-q)}{2}\right)_{k}
\left(\frac{1-(n-1)(q-1)}{2}\right)_{k}
\geq 0.
$$

Now, we infer from \eqref{eq-1.7}, \eqref{eq-1.8}, \eqref{eq-2.11}, \eqref{eq-2.12} and Claim \ref{claim-3.2} that
\begin{eqnarray}\label{eq-3.2}
&\;\;&\frac{\sqrt{\pi} C_{\mathbb{B}^{n},q}(\rho l_{\alpha}; e_{n})}{\Gamma(\frac{n}{2}) (1+\rho^{2})^{(n-1)(q-1)}}\\\nonumber
&\leq&
\sum_{k=0}^{\infty}
\sum_{j=0}^{k}
\frac{(\frac{(n-1)(1-q)}{2})_{k}(\frac{1-(n-1)(q-1)}{2})_{k}}{\Gamma(k+\frac{q+n}{2})}
\left(\frac{2\rho}{1+\rho^{2}}\right)^{2k}\frac{4^{k}\Gamma(k+\frac{q+1}{2} )}
{(2k)!}\\\nonumber
&&\times\left(
\begin{array}{c}
k\\
j\\
\end{array}
\right)
\sin^{2k-2j}\alpha\cos^{2j}\alpha\\\nonumber
&=&
\sum_{k=0}^{\infty}
\frac{4^{k}(\frac{(n-1)(1-q)}{2})_{k}(\frac{1-(n-1)(q-1)}{2})_{k}\Gamma(k+\frac{q+1}{2} )}{(2k)!\Gamma(k+\frac{q+n}{2})}
\left(\frac{2\rho}{1+\rho^{2}}\right)^{2k}\\\nonumber
&=&\frac{\Gamma(\frac{q+1}{2})}
{\Gamma(\frac{q+n}{2})}\;_{3}F_{2}\left(\frac{(n-1)(1-q)}{2},\frac{1-(n-1)(q-1)}{2},\frac{q+1}{2};\frac{1}{2},\frac{q+n}{2};
\frac{4\rho^{2}}{(1+\rho^{2})^{2}}\right).
\end{eqnarray}
The equality occurs if and only if $\alpha=0$ and $\alpha=\pi$,
which shows
\begin{eqnarray}\label{eq-3.3}
C_{\mathbb{B}^{n},q}(\rho l_{\alpha}; e_{n})
\leq C_{\mathbb{B}^{n},q}(\pm\rho e_{n}; e_{n}).
\end{eqnarray}
Similarly, we get from \eqref{eq-2.12} and Claim \ref{claim-3.2} that
\[\begin{split}
&\frac{\sqrt{\pi} C_{\mathbb{B}^{n},q}(\rho l_{\alpha}; e_{n})}{\Gamma(\frac{n}{2}) (1+\rho^{2})^{(n-1)(q-1)}}\\
\geq&
\sum_{k=0}^{\infty}
\sum_{j=0}^{k}
\frac{(\frac{(n-1)(1-q)}{2})_{k}(\frac{1-(n-1)(q-1)}{2})_{k}}{\Gamma(k+\frac{q+n}{2})}
\left(\frac{2\rho}{1+\rho^{2}}\right)^{2k}\frac{\Gamma( \frac{q+1}{2} )}
{k!}
\left(
\begin{array}{c}
k\\
j\\
\end{array}
\right)
\sin^{2k-2j}\alpha\cos^{2j}\alpha\\
=&
\sum_{k=0}^{\infty}
\frac{(\frac{(n-1)(1-q)}{2})_{k}(\frac{1-(n-1)(q-1)}{2})_{k}\Gamma( \frac{q+1}{2} )}{k!\Gamma(k+\frac{q+n}{2})}
\left(\frac{2\rho}{1+\rho^{2}}\right)^{2k},
\end{split}\]
which, together with \eqref{eq-1.7} and \eqref{eq-1.8}, yields
\begin{eqnarray}\label{eq-3.4}
&\;\;&\frac{\sqrt{\pi} C_{\mathbb{B}^{n},q}(\rho l_{\alpha}; e_{n})}{\Gamma(\frac{n}{2}) (1+\rho^{2})^{(n-1)(q-1)}}\\\nonumber
&\geq&\frac{\Gamma(\frac{q+1}{2})}
{\Gamma(\frac{q+n}{2})}\;_{2}F_{1}\left(\frac{(n-1)(1-q)}{2},\frac{1-(n-1)(q-1)}{2};\frac{q+n}{2};
\frac{4\rho^{2}}{(1+\rho^{2})^{2}}\right).
\end{eqnarray}
The equality occurs if and only if $\alpha=\frac{\pi}{2}$,
which implies
\begin{eqnarray}\label{eq-3.5}
C_{\mathbb{B}^{n},q}(\rho l_{\alpha}; e_{n})
\geq C_{\mathbb{B}^{n},q}(\rho e_{n-1}; e_{n})
=C_{\mathbb{B}^{n},q}(\rho e_{n}; e_{n-1}).
\end{eqnarray}

Then for any $x\in\mathbb{B}^{n}\backslash \{0\}$, $l\in\mathbb{S}^{n-1}$, and $t_{x}\in\mathbb{S}^{n-1}$ with
$\langle t_{x},\frac{x}{|x|}\rangle=0$,
by \eqref{eq-1.6}, \eqref{eq-3.1}, \eqref{eq-3.3}, \eqref{eq-3.5} and Lemma \ref{lem-2.2},
we see that
\begin{eqnarray*}
\mathbf{C}_{\mathbb{B}^{n},q}\big(x;t_{x}\big)
&=&\mathbf{C}_{\mathbb{B}^{n},q}(|x| e_{n}; e_{n-1})
\leq \mathbf{C}_{\mathbb{B}^{n},q}\big(x;l\big)
=\mathbf{C}_{\mathbb{B}^{n},q}(|x| l_{\alpha};  e_{n})\\\nonumber
&\leq& \mathbf{C}_{\mathbb{B}^{n},q}(|x| e_{n}; \pm e_{n})
=\mathbf{C}_{\mathbb{B}^{n},q}\big(x;\pm\frac{x}{|x|}\big)
=\mathbf{C}_{\mathbb{B}^{n},q}(x).
\end{eqnarray*}
In particular, for $x\in\mathbb{B}^{n}\backslash \{0\}$, we have
$$
\mathbf{C}_{\mathbb{B}^{n},q}\big(x;t_{x}\big)
<\mathbf{C}_{\mathbb{B}^{n},q}\big(x;\pm\frac{x}{|x|}\big).
$$
The proof of Theorem \ref{thm-1.2}(1) is complete.

\smallskip

 (II) By \eqref{eq-1.4}, \eqref{eq-1.6}, \eqref{eq-3.1}  and \eqref{eq-3.2},
we see that Theorem \ref{thm-1.5}(1) is true.
\qed

 \subsection{Proof of Theorem \ref{thm-1.3}(1)}\label{sec-3.2}
For any $x\in\mathbb{B}^{n}$ and $l\in\mathbb{S}^{n-1}$,
by \eqref{eq-3.1} and letting $q=1$ in \eqref{eq-3.2} and \eqref{eq-3.4}, respectively, we get
$$
\mathbf{C}_{\mathbb{B}^{n},1}(x)\equiv
\mathbf{C}_{\mathbb{B}^{n},1}(x;l)\equiv
\frac{ 2(n-1) \Gamma(\frac{n}{2}) }{\sqrt{\pi} \Gamma(\frac{n+1}{2})(1-|x|^2)}.
$$
Similarly, by \eqref{eq-3.1} and letting $q= \frac{n}{n-1}$ in \eqref{eq-3.2} and \eqref{eq-3.4}, respectively, we obtain
$$
\mathbf{C}_{ \mathbb{B}^{n},\frac{n}{n-1}}(x)\equiv
\mathbf{C}_{ \mathbb{B}^{n},\frac{n}{n-1}}(x;l)
\equiv
\frac{ 2(n-1) }{(1-|x|^2)^{\frac{2n-1}{n}}}
\left(\frac{ \Gamma(\frac{n}{2})  \Gamma(\frac{2n-1}{2n-2}) }{ \sqrt{\pi} \Gamma(\frac{n^{2}}{2n-2}) }(1+|x|^2)\right)^{\frac{n-1}{n}}.
$$
Hence, Theorem \ref{thm-1.3}(1) follows.
 \qed

\subsection{Proofs of Theorems \ref{thm-1.1}(1) and \ref{thm-1.4}(1)}
(I) First, we prove Theorem \ref{thm-1.1}(1).
Since  $q\in(1,\frac{n}{n-1})$,
we see that for any $k\in \mathbb{Z}^{+}$,
$$
\left(\frac{(n-1)(1-q)}{2}\right)_{k}
\left(\frac{1-(n-1)(q-1)}{2}\right)_{k}
<0.
$$
For any $\rho\in[0,1)$ and $\alpha\in[0,\pi]$,
similar arguments as in the proofs of \eqref{eq-3.2} and  \eqref{eq-3.4} guarantee that
\begin{eqnarray}\label{eq-3.6}
&\;\;&\frac{\sqrt{\pi} C_{\mathbb{B}^{n},q}(\rho l_{\alpha}; e_{n})}{\Gamma(\frac{n}{2}) (1+\rho^{2})^{(n-1)(q-1)}}\\\nonumber
&\geq&\frac{\Gamma(\frac{q+1}{2})}
{\Gamma(\frac{q+n}{2})}\;_{3}F_{2}\left(\frac{(n-1)(1-q)}{2},\frac{1-(n-1)(q-1)}{2},\frac{q+1}{2};\frac{1}{2},\frac{q+n}{2};
\frac{4\rho^{2}}{(1+\rho^{2})^{2}}\right)
\end{eqnarray}
and
\begin{eqnarray}\label{eq-3.7}
&\;\;&\frac{\sqrt{\pi} C_{\mathbb{B}^{n},q}(\rho l_{\alpha}; e_{n})}{\Gamma(\frac{n}{2}) (1+\rho^{2})^{(n-1)(q-1)}}\\\nonumber
&\leq& \frac{\Gamma(\frac{q+1}{2})}
{\Gamma(\frac{q+n}{2})}\;_{2}F_{1}\left(\frac{(n-1)(1-q)}{2},\frac{1-(n-1)(q-1)}{2};\frac{q+n}{2};
\frac{4\rho^{2}}{(1+\rho^{2})^{2}}\right).
\end{eqnarray}
The equality holds in \eqref{eq-3.6}  for $\alpha=0$ and $\alpha=\pi$.
 The equality holds in \eqref{eq-3.7} for $\alpha=\frac{\pi}{2}$.
Therefore,
\begin{eqnarray}\label{eq-3.8}
\qquad
C_{\mathbb{B}^{n},q}(\rho l_{\alpha}; e_{n})
\geq C_{\mathbb{B}^{n},q}(\rho e_{n}; \pm e_{n})
\quad\text{and}\quad
C_{\mathbb{B}^{n},q}(\rho l_{\alpha}; e_{n})
\leq C_{\mathbb{B}^{n},q}(\rho e_{n}; e_{n-1}).
\end{eqnarray}

Then for any $x\in\mathbb{B}^{n}\backslash \{0\}$,  $l\in\mathbb{S}^{n-1}$ and $t_{x}\in\mathbb{S}^{n-1}$ with
$\langle t_{x},\frac{x}{|x|}\rangle=0$,
  we obtain from \eqref{eq-1.6}, \eqref{eq-3.1}, \eqref{eq-3.8} and
Lemma \ref{lem-2.2} that
\begin{eqnarray*}
\mathbf{C}_{\mathbb{B}^{n},q}\big(x;\pm\frac{x}{|x|}\big)
&=&\mathbf{C}_{\mathbb{B}^{n},q}\big(|x| e_{n}; \pm e_{n}\big)
\leq
\mathbf{C}_{\mathbb{B}^{n},q}\big(x;l\big)
= \mathbf{C}_{\mathbb{B}^{n},q}(|x| l_{\alpha}; e_{n})\\
&\leq&\mathbf{C}_{\mathbb{B}^{n},q}(|x| e_{n}; e_{n-1})
= \mathbf{C}_{\mathbb{B}^{n},q}\big(x;t_{x}\big)
=\mathbf{C}_{\mathbb{B}^{n},q}(x).
\end{eqnarray*}
In particular, for any $x\in\mathbb{B}^{n}\backslash \{0\}$, we have
$$\mathbf{C}_{\mathbb{B}^{n},q}\big(x;\pm\frac{x}{|x|}\big)
<\mathbf{C}_{\mathbb{B}^{n},q}\big(x;t_{x}\big).$$
The proof of Theorem \ref{thm-1.1}(1) is complete.

\smallskip

(II) By \eqref{eq-1.4}, \eqref{eq-1.6}, \eqref{eq-3.1}  and \eqref{eq-3.7},
we see that Theorem \ref{thm-1.4}(1) is true.
\qed

\section{Proofs of Theorems \ref{thm-1.6} and \ref{thm-1.7} for the unit ball case}\label{sec-4}
The aim of this section is to prove Theorems \ref{thm-1.6}(1) and \ref{thm-1.7}(1).
First, we establish a representation for $\mathbf{C}_{\mathbb{B}^{n},\infty}(x;l)$.

\begin{lemma}\label{lem-4.1}
For any $x\in \mathbb{B}^{n}$ and $l \in \mathbb{S}^{n-1}$,
we have
\begin{eqnarray*}
\mathbf{C}_{\mathbb{B}^{n},\infty}(x;l)
=\sup_{\zeta\in\mathbb{S}^{n-1}} |\langle\nabla \mathcal{P}_{\mathbb{B}^{n}}(x,\zeta),l\rangle|.
\end{eqnarray*}
\end{lemma}

\begin{proof}
Let $u=\mathcal{P}_{\mathbb{B}^{n}}[\phi]$ in $\mathbb{B}^{n}$, where $\phi\in L^{1}(\mathbb{S}^{n-1},\mathbb{R})$.
For any $x\in \mathbb{B}^{n}$ and $l\in \mathbb{S}^{n-1}$, it follows from
\eqref{eq-2.3} that
\begin{eqnarray*}
 |\langle\nabla u(x),l\rangle|\leq
 \sup_{\zeta\in\mathbb{S}^{n-1}} |\langle\nabla \mathcal{P}_{\mathbb{B}^{n}}(x,\zeta),l\rangle|\cdot \|\phi\|_{L^{1}(\mathbb{S}^{n-1},\mathbb{R})},
\end{eqnarray*}
which means
\begin{eqnarray}\label{eq-4.1}
\mathbf{C}_{\mathbb{B}^{n},\infty}(x;l)
\leq\sup_{\zeta\in\mathbb{S}^{n-1}} |\langle\nabla \mathcal{P}_{\mathbb{B}^{n}}(x,\zeta),l\rangle|.
\end{eqnarray}

Next, we show the sharpness of \eqref{eq-4.1}.
For any $x\in \mathbb{B}^{n}$ and $l\in \mathbb{S}^{n-1}$, by \eqref{eq-2.2}, we have
\begin{eqnarray}\label{eq-4.2}
&\;\;& \langle\nabla \mathcal{P}_{\mathbb{B}^{n}}(x,\zeta),l\rangle=-2(n-1)(1-|x|^2)^{n-2}
\frac{ \big\langle x|x-\zeta|^{2}+(1-|x|^{2})(x-\zeta),l\big\rangle }{|x-\zeta|^{2n}}.
\end{eqnarray}
Obviously, the mapping $(x,l,\zeta)\mapsto \langle\nabla \mathcal{P}_{\mathbb{B}^{n}}(x,\zeta),l\rangle$ is continuous in $\mathbb{B}^{n}\times\mathbb{S}^{n-1}\times\mathbb{S}^{n-1} $.
 Then for any $x\in \mathbb{B}^{n}$ and $l\in \mathbb{S}^{n-1}$,
there exists $\zeta^{*}=\zeta^{*}(x,l)\in\mathbb{S}^{n-1}$ such that
\begin{eqnarray} \label{eq-4.3}
&\;\;&\max_{\zeta\in\mathbb{S}^{n-1}}  \langle\nabla \mathcal{P}_{\mathbb{B}^{n}}(x,\zeta),l\rangle
= \langle\nabla \mathcal{P}_{\mathbb{B}^{n}}(x,\zeta^{*}),l\rangle,
\end{eqnarray}
where $\zeta^{*}(x,l)$ means that the point $\zeta^{*}$ depends only on $x$ and $l$.

For $i\in\mathbb{Z}^{+}$, $\zeta\in\mathbb{S}^{n-1}$ and $x\in\mathbb{B}^{n}$, we let
$$
\phi_{i}(\zeta)=\frac{ \chi_{\Omega_{i}}(\zeta)}{ || \chi_{\Omega_{i}} ||_{L^{1}(\mathbb{S}^{n-1},\mathbb{R})}}
\quad\text{and}\quad
u_{i}(x)=\mathcal{P}_{\mathbb{B}^{n}}[\phi_{i}](x),
$$
where $\Omega_{i}=\{\zeta\in\mathbb{S}^{n-1}:|\zeta-\zeta^{*}|\leq\frac{1}{i}\}$
and $\chi$ is the indicator function.
Obviously, for any $i\in \mathbb{Z}^{+}$,
$ ||\phi_{i}||_{L^{1}(\mathbb{S}^{n-1},\mathbb{R}) }=1$ and
\begin{eqnarray} \label{eq-4.4}
 \langle\nabla u_{i}(x),l\rangle
 = \int_{\mathbb{S}^{n-1}}
 \langle\nabla \mathcal{P}_{\mathbb{B}^{n}}(x,\zeta),l\rangle
 \frac{ \chi_{\Omega_{i}}(\zeta)}{ || \chi_{\Omega_{i}} ||_{L^{1}(\mathbb{S}^{n-1},\mathbb{R})}}
  d\sigma(\zeta).
  \end{eqnarray}

\begin{claim}\label{claim-4.1} For any $x\in \mathbb{B}^{n}$ and $l\in \mathbb{S}^{n-1}$,
$$
\lim_{i\rightarrow\infty} \int_{\mathbb{S}^{n-1}}
\langle\nabla \mathcal{P}_{\mathbb{B}^{n}}(x,\zeta),l\rangle
 \cdot
 \frac{ \chi_{\Omega_{i}}(\zeta)}{ || \chi_{\Omega_{i}} ||_{L^{1}(\mathbb{S}^{n-1},\mathbb{R})}}
  d\sigma(\zeta)
 =\max_{\zeta\in\mathbb{S}^{n-1}}  \langle\nabla \mathcal{P}_{\mathbb{B}^{n}}(x,\zeta),l\rangle.
$$
\end{claim}
By the definition of $\chi_{\Omega_{i}} $ and the continuity of the mapping $(x,l,\zeta)\mapsto \langle\nabla \mathcal{P}_{\mathbb{B}^{n}}(x,\zeta),l\rangle$, we obtain
$$
 \lim_{i\rightarrow\infty}
\big( \langle\nabla \mathcal{P}_{\mathbb{B}^{n}}(x,\zeta),l\rangle-
 \langle\nabla \mathcal{P}_{\mathbb{B}^{n}}(x,\zeta^{*}),l\rangle \big)\cdot
 \chi_{\Omega_{i}}(\zeta)
 =0.
$$
Then for any $\varepsilon>0$, there exists a positive integer $m_{1}=m_{1}(\varepsilon)$ such that for any $i\geq m_{1} $,
$$
\big|\langle\nabla \mathcal{P}_{\mathbb{B}^{n}}(x,\zeta),l\rangle-
 \langle\nabla \mathcal{P}_{\mathbb{B}^{n}}(x,\zeta^{*}),l\rangle \big|\cdot
 \chi_{\Omega_{i}}(\zeta)<\varepsilon.
$$
Since
$\int_{\mathbb{S}^{n-1}}
\frac{ \chi_{\Omega_{i}}(\zeta)}{||\chi_{\Omega_{i}}||_{L^{1}(\mathbb{S}^{n-1},\mathbb{R})}}
d\sigma(\zeta)=1$, then for any $i\geq m_{1} $,
\begin{eqnarray*}
&\;\;&\left| \int_{\mathbb{S}^{n-1}}
\langle\nabla \mathcal{P}_{\mathbb{B}^{n}}(x,\zeta),l\rangle
 \frac{ \chi_{\Omega_{i}}(\zeta)}{ || \chi_{\Omega_{i}} ||_{L^{1}(\mathbb{S}^{n-1},\mathbb{R})}}
  d\sigma(\zeta)
 - \langle\nabla \mathcal{P}_{\mathbb{B}^{n}}(x,\zeta^{*}),l\rangle\right|\\
&\leq&\int_{\mathbb{S}^{n-1}}
  \left|\big(\langle\nabla \mathcal{P}_{\mathbb{B}^{n}}(x,\zeta),l\rangle
-\langle\nabla \mathcal{P}_{\mathbb{B}^{n}}(x,\zeta^{*}),l\rangle\big)\cdot\chi_{\Omega_{i}}(\zeta)\right|
 \cdot
 \frac{ \chi_{\Omega_{i}}(\zeta)}{ || \chi_{\Omega_{i}} ||_{L^{1}(\mathbb{S}^{n-1},\mathbb{R})}}d\sigma(\zeta)
 \leq\varepsilon,
\end{eqnarray*}
which, together with \eqref{eq-4.3}, yields that Claim \ref{claim-4.1} is true.
\smallskip

Now, it follows from \eqref{eq-4.4} and Claim \ref{claim-4.1} that
$$
 \lim_{i\rightarrow\infty} \langle\nabla u_{i}(x),l\rangle
=\max_{\zeta\in\mathbb{S}^{n-1}}  \langle\nabla \mathcal{P}_{\mathbb{B}^{n}}(x,\zeta),l\rangle\cdot \lim_{i\rightarrow\infty}\|\phi_{i}\|_{L^{1}(\mathbb{S}^{n-1},\mathbb{R})}.
$$
This, together with \eqref{eq-4.1}, shows that for any $x\in \mathbb{B}^{n}$ and $l\in \mathbb{S}^{n-1}$,
\begin{eqnarray*}
\mathbf{C}_{\mathbb{B}^{n},\infty}(x;l)
=\max_{\zeta\in\mathbb{S}^{n-1}}  \langle\nabla\mathcal{P}_{\mathbb{B}^{n}}(x,\zeta),l\rangle,
\end{eqnarray*}
which is what we need.
 \end{proof}

For any $\alpha\in[0,\pi]$ and $\rho\in[0,1)$,
we let $l_{\alpha}=\sin\alpha \cdot e_{n-1}+\cos\alpha \cdot e_{n}$
and
\begin{eqnarray}\label{eq-4.5}
C_{\mathbb{B}^{n},\infty}(\rho e_{n};l_{\alpha})
=\max_{\zeta\in\mathbb{S}^{n-1}}
\frac{\big|\frac{1-\rho^{2}}{1+\rho^2}\sin \alpha \cdot\zeta_{n-1} + (\zeta_{n} -\frac{2\rho}{1+\rho^{2}})\cos\alpha\big|}{|\rho e_{n}-\zeta|^{2n}},
\end{eqnarray}
where $\zeta=(\zeta_{1},\zeta_{2},\ldots,\zeta_{n})\in \mathbb{S}^{n-1}$.
Then, we obtain the following result.

 \begin{lemma}\label{lem-4.2}
For any $\alpha\in[0,\pi]$ and $\rho\in[0,1)$,
we have
\begin{eqnarray}\label{eq-4.6}
\mathbf{C}_{\mathbb{B}^{n},\infty}(\rho e_{n};l_{\alpha})
=2(n-1)(1+\rho^{2})(1-\rho^2)^{n-2}C_{\infty}(\rho e_{n};l_{\alpha}).
\end{eqnarray}

\end{lemma}

\begin{proof}
For any $x\in \mathbb{B}^{n}$ and $l\in \mathbb{S}^{n-1}$, by \eqref{eq-4.2}, we have
\begin{eqnarray}\label{eq-4.7}
&\;\;&\sup_{\zeta\in\mathbb{S}^{n-1}} |\langle\nabla \mathcal{P}_{\mathbb{B}^{n}}(x,\zeta),l\rangle|\\\nonumber
&=&2(n-1)(1-|x|^2)^{n-2}\max_{\zeta\in\mathbb{S}^{n-1}}
\frac{\big|\big\langle x|x-\zeta|^{2}+(1-|x|^{2})(x-\zeta),l\big\rangle \big|}{|x-\zeta|^{2n}}.
\end{eqnarray}
Let $x=\rho e_{n}$ and $l=l_{\alpha}$, where $\rho\in[0,1)$.
By calculations, we get
\begin{eqnarray*}
&\;\;& \max_{\zeta\in\mathbb{S}^{n-1}}
\frac{\big|\big\langle x|x-\zeta|^{2}+(1-|x|^{2})(x-\zeta),l\big\rangle\big|}{|x-\zeta|^{2n}}\\
 &=&(1+\rho^{2})\max_{\zeta\in\mathbb{S}^{n-1}}
\frac{\big|\frac{1-\rho^{2}}{1+\rho^2}\sin \alpha \cdot\zeta_{n-1} + (\zeta_{n} -\frac{2\rho}{1+\rho^{2}})\cos\alpha\big|}{|\rho e_{n}-\zeta|^{2n}}.
\end{eqnarray*}
This, together with Lemma \ref{lem-4.1}, \eqref{eq-4.5} and \eqref{eq-4.7}, implies that \eqref{eq-4.6} holds true.
\end{proof}

\begin{lemma}\label{lem-4.3}
For any $x\in\mathbb{B}^{n}$, $l\in \mathbb{S}^{n-1}$ and unitary transformation $A$ in $\mathbb{R}^{n}$,
we have
\begin{eqnarray*}
\mathbf{C}_{\mathbb{B}^{n},\infty}(x;l)=
\mathbf{C}_{\mathbb{B}^{n},\infty}(Ax;Al).
\end{eqnarray*}
\end{lemma}

\begin{proof}
For any $x\in\mathbb{B}^{n}$, $l\in \mathbb{S}^{n-1}$ and unitary transformation $A$ in $ \mathbb{R}^{n}$, it follows from Lemma \ref{lem-4.1} and \eqref{eq-4.7} that
\begin{eqnarray*}
&&\mathbf{C}_{\mathbb{B}^{n},\infty}(Ax;Al)\\
&=&2(n-1)(1-|x|^2)^{n-2}\max_{\zeta\in\mathbb{S}^{n-1}}
\frac{\big|\big\langle Ax|Ax-\zeta|^{2}+(1-|x|^{2})(Ax-\zeta),Al\big\rangle\big|}{|Ax-\zeta|^{2n}}.
\end{eqnarray*}
Let $\xi=A^{-1}\zeta$.
By the fact
$|A\xi-Ax|=|\xi-x|$ and $\langle Ax,Al\rangle=\langle x,l\rangle$,
 we find
\[\begin{split}
 \mathbf{C}_{\mathbb{B}^{n},\infty}(Ax;Al)
=&\;2(n-1)(1-|x|^2)^{n-2}\max_{\xi\in\mathbb{S}^{n-1}}
\frac{\big|\big\langle Ax|x-\xi|^{2}+(1-|x|^{2})(Ax-A\xi),Al\big\rangle\big|}{|x-\xi|^{2n}}\\
=&\;2(n-1)(1-|x|^2)^{n-2}\max_{\xi\in\mathbb{S}^{n-1}}
\frac{\big|\big\langle x|x-\xi|^{2}+(1-|x|^{2})(x-\xi),l\big\rangle\big|}{|x-\xi|^{2n}}\\
=&\; \mathbf{C}_{\mathbb{B}^{n},\infty}(x;l),
\end{split}\]
as required.
\end{proof}

\subsection{Proofs of Theorems \ref{thm-1.6}(1) and \ref{thm-1.7}(1)}
(I) First, we prove Theorem \ref{thm-1.6}(1).
For any $ x\in\mathbb{B}^{n}$ and $l\in\mathbb{S}^{n-1}$,
  we choose an unity transformation $A$ in $\mathbb{R}^{n}$
 such that $Ax=\rho e_{n}$ and $Al=l_{\alpha}$,
where $\rho=|x|\in[0,1)$, $l_{\alpha}=\sin\alpha \cdot e_{n-1}+\cos\alpha \cdot e_{n}$ and $\alpha\in[0,\pi]$.
For any $x\in \mathbb{B}^{n}$,
by \eqref{eq-1.6}, Lemmas \ref{lem-4.2} and \ref{lem-4.3}, we get
\begin{eqnarray}\label{eq-4.8}
\;\;\;\;\;\;\mathbf{C}_{\mathbb{B}^{n},\infty}(x)
&=& \sup_{l\in\mathbb{S}^{n-1}}
\mathbf{C}_{\mathbb{B}^{n},\infty}(x;l)
=\sup_{l\in\mathbb{S}^{n-1}}
\mathbf{C}_{ \mathbb{B}^{n},\infty}(Ax;Al)
 =\sup_{\alpha\in[0,\pi]}
  \mathbf{C}_{ \mathbb{B}^{n},\infty}(\rho e_{n};l_{\alpha})
\\\nonumber
 &=&2(n-1)(1+\rho^{2})(1-\rho^2)^{n-2} \sup_{\alpha\in[0,\pi]}C_{\mathbb{B}^{n},\infty}(\rho e_{n};l_{\alpha}).
\end{eqnarray}
Hence, to prove the theorem, it remains to estimate the quantity
$ \sup_{\alpha\in[0,\pi]}C_{\mathbb{B}^{n},\infty}(\rho e_{n};l_{\alpha})$.
\begin{claim}\label{claim-4.2}
 For any $\alpha\in[0,\pi]$ and $\rho\in[0,1)$,
$$\sup_{\alpha\in[0,\pi]}C_{\mathbb{B}^{n},\infty}(\rho e_{n};l_{\alpha})
=C_{\mathbb{B}^{n},\infty}(\rho e_{n};\pm e_{n})
 = \frac{1}{(1+\rho^2)(1-\rho)^{2n-2}}.
 $$
\end{claim}

Using \eqref{eq-4.5} and
 spherical coordinate transformation
(cf. \cite[Equation (2.2)]{chen2018}),
we find
\begin{eqnarray*}
C_{\mathbb{B}^{n},\infty}(\rho e_{n};l_{\alpha})
 &=&\max_{\beta\in[0,\pi] }\max_{ \gamma\in[0,\pi]}
\frac{\big|\frac{1-\rho^{2}}{1+\rho^2}\sin \alpha  \sin\beta\cos\gamma + (\cos\beta -\frac{2\rho}{1+\rho^{2}})\cos\alpha\big|}{(1+\rho^{2}-2\rho\cos\beta)^{n}}.
\end{eqnarray*}
Since the maximum in $\gamma$ is attained either at $\gamma=0$ or $\gamma=\pi$, we get
\begin{eqnarray}\label{eq-4.9}
C_{\mathbb{B}^{n},\infty}(\rho e_{n};l_{\alpha})
 &=&\max_{\beta\in[0,2\pi] }
\frac{\big|\frac{1-\rho^{2}}{1+\rho^2}\sin \alpha  \sin\beta  + (\cos\beta -\frac{2\rho}{1+\rho^{2}})\cos\alpha\big|}{(1+\rho^{2}-2\rho\cos\beta)^{n}}.
\end{eqnarray}
Therefore,
\begin{eqnarray}\label{eq-4.10}
\sup_{\alpha\in[0,\pi]}C_{\mathbb{B}^{n},\infty}(\rho e_{n};l_{\alpha})
&\geq& C_{\mathbb{B}^{n},\infty}(\rho e_{n};\pm e_{n})
=\max_{\beta\in[0,2\pi] }
\frac{\big|  \cos\beta -\frac{2\rho}{1+\rho^{2}} \big|}{(1+\rho^{2}-2\rho\cos\beta)^{n}}\\\nonumber
&\geq&\frac{\big|1 -\frac{2\rho}{1+\rho^{2}} \big|}{(1-\rho)^{2n}}
=\frac{1}{(1+\rho^{2})(1-\rho)^{2n-2}}.
\end{eqnarray}

On the other hand, by \eqref{eq-4.9} and Cauchy–Schwarz inequality, we have
\begin{eqnarray*}
C_{\mathbb{B}^{n},\infty}(\rho e_{n};l_{\alpha})
 &\leq&
 \frac{1}{(1-\rho)^{2n-2}}
 \max_{\beta\in[0,2\pi] }
\left(\frac{\big(\frac{1-\rho^{2}}{1+\rho^2}  \sin\beta )^{2} + (\cos\beta -\frac{2\rho}{1+\rho^{2}})^{2}}{ (1+\rho^{2}-2\rho\cos\beta)^{2}} \right)^{\frac{1}{2}}.
\end{eqnarray*}
Since
$$
\frac{\big(\frac{1-\rho^{2}}{1+\rho^2}  \sin\beta )^{2} + (\cos\beta -\frac{2\rho}{1+\rho^{2}})^{2}}{ (1+\rho^{2}-2\rho\cos\beta)^{2}}
=\frac{1}{(1+\rho^2)^2},
$$
we have
\begin{eqnarray*}
C_{\mathbb{B}^{n},\infty}(\rho e_{n};l_{\alpha})
 &\leq&
 \frac{1}{(1+\rho^2)(1-\rho)^{2n-2}} .
\end{eqnarray*}
This, together with \eqref{eq-4.10}, implies
\begin{eqnarray*}
\sup_{\alpha\in[0,\pi]}C_{\mathbb{B}^{n},\infty}(\rho e_{n};l_{\alpha})
=C_{\mathbb{B}^{n},\infty}(\rho e_{n};\pm e_{n})
 = \frac{1}{(1+\rho^2)(1-\rho)^{2n-2}},
\end{eqnarray*}
and so, the claim is true.

\smallskip

By \eqref{eq-4.8}, Claim \ref{claim-4.2} and Lemma \ref{lem-4.3}, we see that for any $x\in \mathbb{B}^{n}\backslash\{0\}$ and $l\in\mathbb{S}^{n-1}$,
\begin{eqnarray}\label{eq-4.11}
\mathbf{C}_{\mathbb{B}^{n},\infty}(x;l)
\leq\mathbf{C}_{\mathbb{B}^{n},\infty}(x;\pm\frac{x}{|x|})
=\mathbf{C}_{\mathbb{B}^{n},\infty}(x)
=2(n-1)
\frac{(1+|x|)^{n-2}}{(1-|x|)^{n}}.
\end{eqnarray}
 The proof of Theorem \ref{thm-1.6}(1) is complete.
\smallskip

(II)
By \eqref{eq-1.4} and \eqref{eq-4.11},
we see that Theorem \ref{thm-1.7}(1) is true.
\qed

\section{Integral representations for $\mathbf{C}_{\mathbb{H}^{n},q}(x;l)$ with $q\in[1,\infty)$}\label{sec-5}

For $w=(w_{1},w_{2})\in \mathbb{H}^{2}$, the gradient estimate
$$
|\nabla U(w)|\leq \frac{2}{\pi w_{2}}\sup_{x\in \mathbb{H}^{2}} |U(x)|
$$
is sharp if $U$ is a bounded harmonic function in $\mathbb{H}^{2}$.
Using the conformal transformation from $\mathbb{D}$ onto $\mathbb{H}^{2}$ given by $w=\frac{i(1+z)}{1-z}$,
one easily transfers the above inequality into the following pointwise optimal estimate:
$$
|\nabla U(z)|\leq \frac{4}{\pi (1-|z|^{2})}\sup_{x\in \mathbb{B}^{2}} |U(x)|,
$$
where this time $U$ is a bounded harmonic function in $ \mathbb{D}$.
 For the above inequalities, we refer to \cite{colo, kr2007, mark}.

Assume that $u=\mathcal{P}_{\mathbb{H}^{n}}[\phi]$ with $\phi\in L^{p}(\mathbb{R}^{n-1},\mathbb{R})$
and $T:\mathbb{B}^{n}\rightarrow\mathbb{H}^{n}$ is the M\"{o}bius transform   given by
$$
T(w)=-e_{n}+\frac{2(w+e_{n})}{|w+e_{n}|^{2}}.
$$
It follows from
\cite[Exercise 3.5.17 and Theorem 5.3.5]{sto2016}
that
$$
\mathcal{P}_{\mathbb{H}^{n}}[\phi](x)
=\mathcal{P}_{\mathbb{B}^{n}}[\phi]\big(T(w)\big)
=\mathcal{P}_{\mathbb{B}^{n}}[\phi\circ T](w),
$$
where $x=T(w)\in \mathbb{H}^{n}$.
However, it seems that this transformation does not provide any fruitful connection between the two sharp constants $\mathbf{C}_{\mathbb{B}^{n},q}(x;l)$ and $\mathbf{C}_{\mathbb{H}^{n},q}(x;l)$.

The main purpose of this section is to establish some general
representations for the
sharp quantity $\mathbf{C}_{\mathbb{H}^{n},q}(x;l)$ when $q\in [1,\infty)$.
Before the proofs, for convenience,
we use $x'$ to denote the point $(x_{1},\ldots,x_{n-1})\in \mathbb{R}^{n-1}$,
where $x=(x_{1},\ldots,x_{n})\in \mathbb{R}^{n}$.
\begin{lemma}\label{lem-5.1}
For $x=(x',x_{n})\in \mathbb{H}^{n}$ and $y=(y',0)\in\mathbb{R}^{n}$,
let $\psi(y)=\frac{y-x}{|y-x|}\in \mathbb{S}_{-}^{n-1}$.
Then we have
 \begin{eqnarray*}
dS_{n-1}\big(\psi(y)\big)
=\frac{x_{n}}{|y-x|^{n}}dV_{n-1}(y'),
\end{eqnarray*}
where
$dS_{n-1}$ is the $(n-1)$-dimensional Lebesgue surface measure.
\end{lemma}

\begin{proof}
Let $x=(x',x_{n})\in \mathbb{H}^{n}$, $y=(y',0)\in\mathbb{R}^{n}$ and
$\psi(y)=\frac{y-x}{|y-x|}\in \mathbb{S}_{-}^{n-1}$.
Then for any $1\leq i\leq n-1$ and $1\leq j\leq n$,
$$
\frac{\partial}{\partial y_{i}}\psi_{j}(y)
=\frac{\delta_{i,j}}{|y-x|}-\frac{(y_{i}-x_{i})(y_{j}-x_{j})}{|y-x|^{3}}
\quad\text{and}\quad
 \frac{\partial}{\partial y_{n}}\psi_{j}(y)=0,
 $$
where
$\psi=(\psi_{1},\ldots,\psi_{n})$ and
$\delta_{i,j}=\left\{
                     \begin{array}{ll}
                       1, & \hbox{if $i=j$,} \\
                       0, & \hbox{if $i\not=j$.}
                     \end{array}
                   \right.$
For any $w\in \mathbb{R}^{n}\backslash\{0\}$,
define
$$
\Phi(w)
=  \left(
\begin{array}{cccc}
w_{1}^{2}&
w_{1} w_{2}& \cdots & 0 \\
w_{2}w_{1}&
w_{2}^{2}&\ddots& \vdots \\
 \vdots & \ddots &  \ddots & 0 \\
w_{n}w_{1}& \cdots &w_{n}w_{n-1} & 0 \\
\end{array}
\right)
$$
and
 \begin{eqnarray*}
\Psi(w)=
\frac{1}{|w|^{4}}\left(
\begin{array}{ccccc}
|w|^{2}- w_{1}^{2} &
w_{1}w_{2}& \cdots &w_{1}w_{n-1}& 0 \\
w_{2}w_{1}&
|w|^{2}- w_{2}^{2}& \cdots &w_{2}w_{n-1}&0\\
\vdots &\vdots & \ddots& \vdots&\vdots \\
w_{n-1}w_{1}& w_{n-1}w_{2}&\cdots & |w|^{2}- w_{n-1}^{2}  & 0 \\
0& 0&\cdots & 0 & 0 \\
\end{array}
\right).
 \end{eqnarray*}
Making elementary calculations, we obtain the Jacobian
matrix
$$
D\psi (y)
=\big(\nabla  \psi_{1} (y)\cdots\nabla  \psi_{n} (y) \big)^{T}
=\left(
\begin{array}{cccc}
\frac{1}{|y-x|} &
0& \cdots & 0 \\
0&
\ddots  &\ddots& \vdots \\
 \vdots &\ddots &  \frac{1}{|y-x|} & 0 \\
0& \cdots & 0 & 0 \\
\end{array}
\right)
-\frac{\Phi(y-x)}{ |y-x|^{3} }
$$
and
$$
\big(D\psi (y)\big)^{T}D\psi (y)=\Psi(y-x),
$$
where $T$ is the transpose and $\nabla  \psi_{i}$ are understood as column vectors.
Since the eigenvalues of $\big(D\psi (y)\big)^{T}D\psi (y)$ are
$$\lambda_{1}^{2}=0,\quad\lambda_{2}^{2}=\cdots\lambda_{n-1}^{2}=\frac{1}{|y-x|^{2}}\quad\text{and}\quad
\lambda_{n}^{2}=\frac{(y_{n}-x_{n})^{2}}{|y-x|^{4}},$$
we see that
$$
dS_{n-1}(\psi(y))
=\frac{|y_{n}-x_{n}|}{|y-x|^{n}}dV_{n-1}(y')
=\frac{ x_{n}}{|y-x|^{n}}dV_{n-1}(y').
$$
The proof of the lemma is complete.
\end{proof}

Based on Lemma \ref{lem-5.1}, we get the following integral representation of $\mathbf{C}_{\mathbb{H}^{n},q}(x,l)$, where $q\in[1,\infty)$.

\begin{lemma}\label{lem-5.2}
For $q\in[1,\infty)$, $x\in\mathbb{H}^{n}$ and $l\in \mathbb{S}^{n-1}$, we have
$$
\mathbf{C}_{\mathbb{H}^{n},q}(x;l)
=\frac{2^{n-2}(n-1) \Gamma(\frac{n}{2})}{ \pi^{\frac{n}{2}}x_{n}^{\frac{n(q-1)+1}{q}}}
\left(\int_{\mathbb{S}_{+}^{n-1}}
 \big|\big\langle e_{n}-2\langle \xi,e_{n} \rangle
\xi,l\big\rangle \big|^{q}
\cdot\langle \xi,e_{n}\rangle^{2(n-1)q-n}
dS_{n-1}(\xi)\right)^{\frac{1}{q}}.
$$

\end{lemma}

\begin{proof}
 Assume that $p$ is the conjugate of $q$ with $p\in(1,\infty]$, $ \phi \in L^{p}(\mathbb{R}^{n-1},\mathbb{R})$
 and $u=\mathcal{P}_{\mathbb{H}^{n}}[\phi]$ in $\mathbb{H}^{n}$.
Let $x=(x',x_{n})\in\mathbb{H}^{n}$ and $y=(y',0)\in\mathbb{R}^{n}$.
For any $i\in\{1,2,\ldots,n-1\}$, by \eqref{eq-1.2}, we obtain that
\begin{eqnarray*}
 \frac{\partial}{\partial x_{i}} u(x)
&=&-2(n-1)c_{n}\int_{\mathbb{R}^{n-1}}\frac{x_{n}^{n-1}(x_{i}-y_{i})}{(|x'-y'|^{2}+x_{n}^{2})^{n}}\phi(y')dV_{n-1}(y')\\\nonumber
&=&2(n-1)c_{n}\int_{\mathbb{R}^{n-1}}\frac{x_{n}^{n-1}(y_{i}-x_{i})}{|y-x|^{2n}}\phi(y')dV_{n-1}(y')
\end{eqnarray*}
and
\begin{eqnarray*}
\frac{\partial}{\partial x_{n}} u(x)
=(n-1)c_{n}\int_{\mathbb{R}^{n-1}}
\left(\frac{x_{n}^{n-2}}{|y-x|^{2n-2}}
+\frac{2x^{n-1}_{n}(y_{n}-x_{n})}{|y-x|^{2n}}\right) \phi(y')dV_{n-1}(y'),
\end{eqnarray*}
where $c_{n}$ is the constant from \eqref{eq-1.3}.
Therefore, for any $x\in\mathbb{H}^{n}$ and $l\in \mathbb{S}^{n-1}$,
\begin{eqnarray}\label{eq-5.1}
\nabla u(x)
=\int_{\mathbb{R}^{n-1}}\nabla \mathcal{P}_{\mathbb{H}^{n}}(x,y') \phi(y')dV_{n-1}(y')
\end{eqnarray}
and
\begin{eqnarray}\label{eq-5.2}
\quad\mathbf{C}_{\mathbb{H}^{n},q}(x;l)
&\leq& \left(\int_{\mathbb{R}^{n-1}}
|\langle\nabla  \mathcal{P}_{\mathbb{H}^{n}}(x,y') ,l\rangle|^{q}dV_{n-1}(y')\right)^{\frac{1}{q}},
\end{eqnarray}
 where
\begin{eqnarray}\label{eq-5.3}
 \nabla  \mathcal{P}_{\mathbb{H}^{n}}(x,y')
 =(n-1)c_{n}
\left(\frac{x_{n}^{n-2}e_{n}}{|y-x|^{2n-2}}
+\frac{2x^{n-1}_{n}(y-x)}{|y-x|^{2n}}\right).
\end{eqnarray}

In order to calculate the quantity $\mathbf{C}_{\mathbb{H}^{n},q}(x;l)$,
 we first estimate the right-hand side of \eqref{eq-5.2}.
Since $x_{n}>0$ and $y_{n}=0$, then $x\not=y$.
Let
$$
e_{xy}=\frac{y-x }{|y-x |}.
$$
By calculations, we deduce
\begin{eqnarray*}
 \frac{2x^{n-1}_{n}(y-x)}{|y-x|^{2n}}
&=&\frac{2x_{n}^{n-2} }{|y-x|^{2n-2}}
 \cdot\frac{x_{n} (y-x)}{|y-x|^{2}}
  =\frac{-2x^{n-2}_{n}}{|y-x|^{2n-2}}
 \left \langle\frac{y-x }{|y-x|},e_{n}\right\rangle
  \frac{y-x }{|y-x|}\\
&=&\frac{-2x^{n-2}_{n}}{|y-x|^{2n-2}}
 \langle e_{xy},e_{n} \rangle
 e_{xy},
\end{eqnarray*}
which implies
\begin{eqnarray}\label{eq-5.4}
 \nabla  \mathcal{P}_{\mathbb{H}^{n}}(x,y')
&=& (n-1)c_{n}x_{n}^{n-2}
\frac{ e_{n}-2\langle e_{xy},e_{n} \rangle
 e_{xy} }{|y-x|^{2n-2}},
\end{eqnarray}
and so,
\begin{eqnarray}\label{eq-5.5}
&& \left(\int_{\mathbb{R}^{n-1}}
|\langle\nabla \mathcal{P}_{\mathbb{H}^{n}}(x,y') ,l\rangle|^{q}dV_{n-1}(y')\right)^{\frac{1}{q}} \\\nonumber
&=&(n-1)c_{n}
\left(\int_{\mathbb{R}^{n-1}}
\frac{x_{n}^{(n-2)q} \big|\big\langle e_{n}-2\langle e_{xy},e_{n} \rangle
 e_{xy},l\big\rangle \big|^{q} }{|y-x|^{2(n-1)q}}
dV_{n-1}(y')\right)^{\frac{1}{q}}.
\end{eqnarray}
 In view of Lemma \ref{lem-5.1}, we get
\begin{eqnarray}\label{eq-5.6}
\;\;
\frac{x_{n}^{(n-2)q} }{|y-x|^{2(n-1)q}}dV_{n-1}(y')
&=&x_{n}^{n(1-q)-1}
\left(\frac{x_{n} }{|y-x| } \right)^{2(n-1)q-n}\frac{x_{n} }{|y-x|^{n} } dV_{n-1}(y')\\\nonumber
&=& x_{n}^{n(1-q)-1}\langle \xi,-e_{n}\rangle^{2(n-1)q-n} dS_{n-1}(\xi).
\end{eqnarray}
Combing \eqref{eq-5.6} and \eqref{eq-5.6} yields that
\begin{eqnarray}\label{eq-5.7}
&& \left(\int_{\mathbb{R}^{n-1}}
|\langle\nabla  \mathcal{P}_{\mathbb{H}^{n}}(x,y') ,l\rangle|^{q}dV_{n-1}(y')\right)^{\frac{1}{q}}\\\nonumber
&=&\frac{(n-1)c_{n}}{x_{n}^{\frac{n(q-1)+1}{q}}}
\left(\int_{\mathbb{S}_{-}^{n-1}}
 \big|\big\langle e_{n}-2\langle \xi,e_{n} \rangle
\xi,l\big\rangle \big|^{q}
\cdot\langle \xi,-e_{n}\rangle^{2(n-1)q-n}
dS_{n-1}(\xi)\right)^{\frac{1}{q}}\\\nonumber
&=&\frac{(n-1)c_{n}}{x_{n}^{\frac{n(q-1)+1}{q}}}
\left(\int_{\mathbb{S}_{+}^{n-1}}
 \big|\big\langle e_{n}-2\langle \xi,e_{n} \rangle
\xi,l\big\rangle \big|^{q}
\cdot\langle \xi,e_{n}\rangle^{2(n-1)q-n}
dS_{n-1}(\xi)\right)^{\frac{1}{q}},
\end{eqnarray}
where $\xi=\frac{y-x}{|y-x|}$.

Next we show the sharpness of \eqref{eq-5.2}.
Fix $l\in \mathbb{S}^{n-1}$ and  $w\in\mathbb{H}^{n}$ with $w=(w_{1},\ldots,w_{n})$.
For any $y'\in \mathbb{R}^{n-1}$ and $x\in\mathbb{H}^{n}$, we define
$$
\phi_{l}(y')=|\langle \nabla \mathcal{P}_{\mathbb{H}^{n}}(w,y'),l\rangle|^{\frac{q}{p}} \cdot \text{sign}\langle \nabla  \mathcal{P}_{\mathbb{H}^{n}}(w,y'),l\rangle
$$
and
$u_{l} (x)=\mathcal{P}_{\mathbb{H}^{n}}[\phi_{l}](x) $.
The similar arguments as above show that
\[\begin{split}
\|\phi_{l}\|_{L^{p}(\mathbb{R}^{n-1},\mathbb{R})}
=&\int_{\mathbb{R}^{n-1}}|\nabla \mathcal{P}_{\mathbb{H}^{n}}(w,y') |^{q}dV_{n-1}(y')\\
=&\frac{(n-1)c_{n}}{w_{n}^{\frac{n(q-1)+1}{q}}}
\left(\int_{\mathbb{S}_{+}^{n-1}}
 \big|  e_{n}-2\langle \xi,e_{n} \rangle
\xi \big|^{q}
\cdot\langle \xi,e_{n}\rangle^{2(n-1)q-n}
dS_{n-1}(\xi)\right)^{\frac{1}{q}}
<\infty.
\end{split}\]
Moreover, \eqref{eq-5.1} yields that
\begin{eqnarray*}
\langle\nabla u_{l}(w),l\rangle
&=& \int_{\mathbb{R}^{n-1}}|\langle\nabla\mathcal{P}_{\mathbb{H}^{n}}(w,y') ,l\rangle|^{q}dV_{n-1}(y')\\
&=&\left(\int_{\mathbb{R}^{n-1}}
|\langle\nabla \mathcal{P}_{\mathbb{H}^{n}}(w,y') ,l\rangle|^{q}dV_{n-1}(y')\right)^{\frac{1}{q}}\|\phi_{l}\|_{L^{p}(\mathbb{R}^{n-1},\mathbb{R})}.
\end{eqnarray*}
Then the arbitrary of $w$ and $l$ shows the sharpness of \eqref{eq-5.2}.
This, together with \eqref{eq-5.7}, implies
\begin{eqnarray*}
\quad\mathbf{C}_{\mathbb{H}^{n},q}(x;l)
&=& \left(\int_{\mathbb{R}^{n-1}}
|\langle\nabla  \mathcal{P}_{\mathbb{H}^{n}}(x,y') ,l\rangle|^{q}dV_{n-1}(y')\right)^{\frac{1}{q}}\\
&=&\frac{(n-1)c_{n}}{x_{n}^{\frac{n(q-1)+1}{q}}}
\left(\int_{\mathbb{S}_{+}^{n-1}}
 \big|\big\langle e_{n}-2\langle \xi,e_{n} \rangle
\xi,l\big\rangle \big|^{q}
\cdot\langle \xi,e_{n}\rangle^{2(n-1)q-n}
dS_{n-1}(\xi)\right)^{\frac{1}{q}},
\end{eqnarray*}
as required.
\end{proof}

For $q\in[1,\infty)$, $x\in\mathbb{H}^{n}$ and $l\in \mathbb{S}^{n-1}$,
let
\begin{eqnarray}\label{eq-5.8}
C_{\mathbb{H}^{n},q}(x;l)
&=& \int_{\mathbb{S}_{+}^{n-1}}
 \big|\big\langle e_{n}-2\langle \xi,e_{n} \rangle
\xi,l\big\rangle \big|^{q}
\cdot\langle \xi,e_{n}\rangle^{2(n-1)q-n}
dS_{n-1}(\xi).
\end{eqnarray}
Thus, in order to estimate $\mathbf{C}_{\mathbb{H}^{n},q}(x;l)$, we only need to calculate $C_{\mathbb{H}^{n},q}(x;l)$.
By spherical coordinate transformation, we can reformulate  $C_{\mathbb{H}^{n},q}(x;l)$ as follows.

 \begin{lemma}\label{lem-5.3}
For $q\in[1,\infty)$, $x\in\mathbb{H}^{n}$ and $l\in \mathbb{S}^{n-1}$,
we have
 \begin{eqnarray*}
C_{\mathbb{H}^{n},q}(x;l)= \frac{1}{2^{(n-1)q}}\int_{\mathbb{S}^{n-1}}|\langle \eta,l\rangle|^{q}\big(1+\langle \eta,e_{n}\rangle\big)^{(n-1)(q-1)}dS_{n-1}(\eta).
\end{eqnarray*}

\end{lemma}
\begin{proof}
Let $\xi=(\xi_{1},\ldots,\xi_{n})\in \mathbb{S}_{+}^{n-1}$ and $l=(l_{1},\ldots,l_{n})\in \mathbb{S}^{n-1}$.
By \eqref{eq-5.8}, we get
\begin{eqnarray*}
C_{\mathbb{H}^{n},q}(x;l)
&=&\int_{\mathbb{S}_{+}^{n-1}}
 \Big|2\xi_{n}\sum_{k=1}^{n-1}\xi_{k}l_{k}+(2\xi_{n}^{2}-1)l_{n} \Big|^{q}
\xi_{n}^{2(n-1)q-n}
dS_{n-1}(\xi) .
\end{eqnarray*}
 Let
 $$\begin{array}{rcl}
\xi_{n} &=&\cos\theta_{1},\\
\xi_{n-1}&=&\sin\theta_{1}\cos\theta_{2},\\
&\vdots &\\
\xi_{2} &=& \sin\theta_{1} \sin\theta_{2} \ldots  \sin\theta_{n-2}\cos\theta_{n-1},\\
\xi_{1} &=& \sin\theta_{1} \sin\theta_{2} \ldots \sin\theta_{n-2}\sin\theta_{n-1},
\end{array}
 $$
 where $\theta_{1}\in[0,\frac{\pi}{2})$, $\theta_{2},\ldots,\theta_{n-2}\in[0,\pi]$ and $\theta_{n-1}\in[0,2\pi]$.
 Then we obtain from \cite[Section 2.2]{chen2018} that
\begin{eqnarray*}
C_{\mathbb{H}^{n},q}(x;l)
 &=&\int_{0}^{\frac{\pi}{2}} \sin^{n-2}\theta_{1}\,d\theta_{1}\int_{0}^{\pi} \sin^{n-3}\theta_{2}\,d\theta_{2}\cdots \int_{0}^{\pi} \sin\theta_{n-2}\,d\theta_{n-2} \\
 &\;\;&\times\int_{0}^{2\pi}h(2\theta_{1},\ldots,\theta_{n-1},l) \, |\cos\theta_{1}|^{2(n-1)q-n}\, d\theta_{n-1},
\end{eqnarray*}
where
\begin{eqnarray*}
&&h(2\theta_{1},\ldots,\theta_{n-1},l)\\
&=&\big|\sin2\theta_{1} \sin\theta_{2} \cdots \sin\theta_{n-2}\sin\theta_{n-1}l_{1}
+\sin2\theta_{1} \sin\theta_{2} \cdots \sin\theta_{n-2}\cos\theta_{n-1}l_{2}\\
&&+\cdots+\sin2\theta_{1}\cos\theta_{2}l_{n-1}
+\cos2\theta_{1}l_{n}\big|^{q}.
\end{eqnarray*}
Let $\vartheta_{1}=2\theta_{1}$ and $\vartheta_{i}= \theta_{i}$ for $i\in\{2,\ldots,n-1\}$.
Therefore,
\begin{eqnarray*}
C_{\mathbb{H}^{n},q}(x;l)
 &=&\frac{1}{2} \int_{0}^{ \pi } \sin^{n-2} \frac{\vartheta_{1}}{2} \,d\vartheta_{1}\int_{0}^{\pi} \sin^{n-3}\vartheta_{2}\,d\vartheta_{2}\cdots \int_{0}^{\pi} \sin\vartheta_{n-2}\,d\vartheta_{n-2} \\
 &\;\;&\times\int_{0}^{2\pi}h(\vartheta_{1},\ldots,\vartheta_{n-1},l) \, \big(\cos\frac{\vartheta_{1}}{2}\big)^{2(n-1)q-n} \, d\vartheta_{n-1}\\
 &=&  \int_{0}^{ \pi } \sin^{n-2}  \vartheta_{1} \,d\vartheta_{1}\int_{0}^{\pi} \sin^{n-3}\vartheta_{2}\,d\vartheta_{2}\cdots \int_{0}^{\pi} \sin\vartheta_{n-2}\,d\vartheta_{n-2} \\
 &\;\;&\times\int_{0}^{2\pi}h(\vartheta_{1},\ldots,\vartheta_{n-1},l) \, \frac{ \big(\cos\frac{\vartheta_{1}}{2}\big)^{2(n-1)(q-1)} }
 {2^{n-1} }\, d\vartheta_{n-1}.
\end{eqnarray*}
Set
 $$
 \begin{array}{rcl}
\eta_{n}  &=& \cos\vartheta_{1},\\
\eta_{n-1} &=&\sin\vartheta_{1}\cos\vartheta_{2},\\
&\vdots &\\
\eta_{2} &=& \sin\vartheta_{1} \sin\vartheta_{2} \ldots  \sin\vartheta_{n-2}\cos\vartheta_{n-1},\\
\eta_{1}  &=& \sin \vartheta_{1} \sin\vartheta_{2} \ldots \sin\vartheta_{n-2}\sin\vartheta_{n-1}.
\end{array}
$$
Then $\eta=(\eta_{1},\ldots,\eta_{n})\in \mathbb{S}^{n-1}\backslash\{-e_{n}\}$,
 $$
|\cos\frac{\vartheta_{1}}{2}|=\sqrt{\frac{1+\eta_{n}}{2}}
\quad\text{and}\quad h(\vartheta_{1},\ldots,\vartheta_{n-1},l)=|\langle\eta,l\rangle|^{q}.
$$

Hence,
\begin{eqnarray*}
C_{\mathbb{H}^{n},q}(x;l)
 &=&\frac{1}{2^{(n-1)q}} \int_{\mathbb{S}^{n-1}}|\langle\eta,l\rangle|^{q}(1+\eta_{n})^{(n-1)(q-1)} dS_{n-1}(\eta).
\end{eqnarray*}
The proof of the lemma is complete.
  \end{proof}

\section{Proofs of Theorems \ref{thm-1.1}$\sim$\ref{thm-1.7} for the half-space case}\label{sec-6}

The aim of this section is to prove Theorems \ref{thm-1.1}$\sim$\ref{thm-1.7}  for the half-space case.

\subsection{Proofs of Theorems \ref{thm-1.6}(2) and \ref{thm-1.7}(2)}
(I) First, we prove Theorem \ref{thm-1.6}(2).
For any $x\in\mathbb{H}^{n}$ and $l\in \mathbb{S}^{n-1}$,
we infer from \eqref{eq-1.5}, \eqref{eq-1.6}  and \eqref{eq-5.1} that
\begin{eqnarray}\label{eq-6.1}
\mathbf{C}_{\mathbb{H}^{n},\infty}(x;l)
&\leq&\sup_{y'\in\mathbb{R}^{n-1}}|\langle\nabla  \mathcal{P}_{\mathbb{H}^{n}}(x,y') ,l\rangle|.
\end{eqnarray}
Let $e_{xy}=\frac{y-x }{|y-x |}$, where $y=(y',0)\in \mathbb{R}^{n}$.
Then \eqref{eq-5.4} and   the fact
$\frac{x_{n}}{|y-x|}=\langle e_{xy},-e_{n}\rangle$ yield that
\begin{eqnarray}\label{eq-6.2}
 |\langle\nabla  \mathcal{P}_{\mathbb{H}^{n}}(x,y') ,l\rangle|
&=&(n-1)c_{n}
\frac{x_{n}^{n-2} \big|\big\langle e_{n}-2\langle e_{xy},e_{n} \rangle
 e_{xy},l\big\rangle\big| }{|y-x|^{2n-2}}\\ \nonumber
&=&\frac{(n-1)c_{n}}{x_{n}^{n}}
\big|\big\langle e_{n}-2\langle e_{xy},e_{n} \rangle
 e_{xy},l\big\rangle\big| \cdot\langle e_{xy},-e_{n}\rangle^{2n-2} ,
\end{eqnarray}
where $c_{n}$ is the constant from \eqref{eq-1.3}.
Note that $ e_{xy}\in\mathbb{S}_{-}^{n-1}$.
Therefore,
\begin{eqnarray}\label{eq-6.3}
 \quad\sup_{y'\in\mathbb{R}^{n-1}}|\langle\nabla \mathcal{P}_{\mathbb{H}^{n}}(x,y') ,l\rangle|
 &=&\frac{(n-1)c_{n}}{x_{n}^{n}}\sup_{\xi\in\mathbb{S}_{-}^{n-1}}
\big|\big\langle e_{n}-2\langle \xi,e_{n} \rangle
\xi,l\big\rangle\big| \cdot\langle \xi,-e_{n}\rangle^{2n-2} \\ \nonumber
 &=&\frac{(n-1)c_{n}}{x_{n}^{n}}\sup_{\xi\in\mathbb{S}_{+}^{n-1}}
\big|\big\langle e_{n}-2\langle \xi,e_{n} \rangle
\xi,l\big\rangle\big| \cdot\langle \xi,e_{n}\rangle^{2n-2}.
\end{eqnarray}
Taking into account the equality
$\big| e_{n}-2\langle \xi,e_{n} \rangle
\xi \big|^{2}
=1$
gives
\begin{eqnarray}\label{eq-6.4}
\sup_{y'\in\mathbb{R}^{n-1}}|\langle\nabla  \mathcal{P}_{\mathbb{H}^{n}}(x,y') ,l\rangle|
&\leq&\frac{(n-1)c_{n}}{x_{n}^{n}} \sup_{\xi\in\mathbb{S}_{+}^{n-1}}
\big| e_{n}-2\langle \xi,e_{n} \rangle
\xi \big| \cdot\langle \xi,e_{n}\rangle^{2n-2}\\\nonumber
&=&\frac{(n-1)c_{n}}{x_{n}^{n}} \sup_{\xi\in\mathbb{S}_{+}^{n-1}}
 \langle \xi,e_{n}\rangle^{2n-2}
 =\frac{(n-1)c_{n}}{x_{n}^{n}},
\end{eqnarray}
which, together with \eqref{eq-1.3} and \eqref{eq-6.1}, implies
\begin{eqnarray} \label{eq-6.5}
\mathbf{C}_{\mathbb{H}^{n},\infty}(x)
=\sup_{l\in \mathbb{S}^{n-1}}\mathbf{C}_{\mathbb{H}^{n},\infty}(x;l)
\leq \frac{(n-1)c_{n}}{x_{n}^{n}}
=\frac{2^{n-2} (n-1)\Gamma(\frac{n}{2})}{ \pi^{\frac{n}{2}}x_{n}^{n}}.
\end{eqnarray}

In the following, we show that the constant in \eqref{eq-6.5} is sharp.
For each $x\in \mathbb{H}^{n}$,
by \eqref{eq-6.2}$\sim$\eqref{eq-6.4}, we see that
\begin{eqnarray} \label{eq-6.6}
\;\;\;\;\;\sup_{l\in\mathbb{S}^{n-1}}\sup_{y'\in\mathbb{R}^{n-1}}
|\langle\nabla  \mathcal{P}_{\mathbb{H}^{n}}(x,y') ,l\rangle|
 =|\langle\nabla \mathcal{P}_{\mathbb{H}^{n}}(x,x') ,\pm e_{n}\rangle|
 =\frac{2^{n-2} (n-1)\Gamma(\frac{n}{2})}{ \pi^{\frac{n}{2}}x_{n}^{n}}.
  \end{eqnarray}
Fix $w=(w',w_{n})\in \mathbb{H}^{n}$.
For any $j\in \mathbb{Z}^{+}$,
we define
$$
\phi_{j}(y')=\frac{ \chi_{\Omega_{j}}(y')}{ || \chi_{\Omega_{j}} ||_{L^{1}(\mathbb{H}^{n},\mathbb{R})}}
$$
in $\mathbb{R}^{n-1}$
and
$u_{j}(x)=\mathcal{P}_{\mathbb{H}^{n}}[\phi_{j}](x) $ in $\mathbb{H}^{n}$,
where
$\Omega_{j}=\{y'\in\mathbb{R}^{n-1}:|y'-w'|\leq\frac{1}{j}\}$.
Then for $j\in \mathbb{Z}^{+}$, $x\in \mathbb{H}^{n}$ and
$l\in \mathbb{S}^{n-1}$,
$ ||\phi_{j}||_{L^{1}(\mathbb{H}^{n},\mathbb{R}) }=1$ and
\begin{eqnarray} \label{eq-6.7}
\langle \nabla u_{j}(x),l\rangle
 = \int_{\mathbb{R}^{n-1}}
 \langle \nabla \mathcal{P}_{\mathbb{H}^{n}}(x,y'),l\rangle
 \frac{ \chi_{\Omega_{j}}(y')}{ || \chi_{\Omega_{j}} ||_{L^{1}(\mathbb{H}^{n},\mathbb{R})}}
  dV_{n-1}(y').
  \end{eqnarray}
For $y',w'\in \mathbb{R}^{n-1}$ and $x\in \mathbb{H}^{n}$,
using \eqref{eq-5.3} and the definition of $\chi_{\Omega_{j}}$,
we find
\[\begin{split}
&
\big| \langle \nabla \mathcal{P}_{\mathbb{H}^{n}}(x,y')  -
\nabla \mathcal{P}_{\mathbb{H}^{n}}(x,w'),l\rangle \big|\cdot
 \chi_{\Omega_{j}}(y')\\
\leq& (n-1)c_{n}x_{n}^{n-2}
\frac{\big||y-x|^{2n-2}-|w-x|^{2n-2}\big|}{|y-x|^{2n-2}|w-x|^{2n-2}}
+2(n-1)c_{n}x^{n-1}_{n}
\left|\frac{ y-x }{|y-x|^{2n}}-\frac{ w-x }{|w-x|^{2n}}\right|,
\end{split}\]
  where $w=(w',0)\in \mathbb{R}^{n}$ and $y=(y',0)\in \mathbb{R}^{n}$ .
Clearly,
  \begin{eqnarray*}
  \frac{\big||y-x|^{2n-2}-|w-x|^{2n-2}\big|}{|y-x|^{2n-2}|w-x|^{2n-2}}
&\leq&\frac{|y-w|}{x_{n}^{4(n-1)}}
\sum_{k=0}^{2n-3} |y-x|^{k}|w-x|^{2n-3-k}\\
&\leq&\frac{(2n-2)|y'-w'|}{x_{n}^{4(n-1)}}
 (|w-x|+|y'-w'|)^{2n-3}
  \end{eqnarray*}
  and
 \begin{eqnarray*}
  &&\left|\frac{ y-x }{|y-x|^{2n}}-\frac{ w-x }{|w-x|^{2n}}\right|\\
 &\leq& \left| \frac{(y-x)(|w-x|^{2n}-|y-x|^{2n}) }{x_{n}^{4n}}
 +\frac{(y-w)|y-x|^{2n} }{x_{n}^{4n}}\right|\\
&\leq&\frac{|y-x|\cdot|y-w|}{x_{n}^{4n}}
\sum_{k=0}^{2n-1} |y-x|^{k}|w-x|^{2n-1-k} +\frac{|y-w|\cdot|y-x|^{2n} }{x_{n}^{4n}}\\
&\leq&\frac{(2n+1)|y'-w'|}{x_{n}^{4n}}
 (|w-x|+|y'-w'|)^{2n}.
  \end{eqnarray*}
Therefore,
$$
\lim_{j\rightarrow\infty}
\big| \langle \nabla \mathcal{P}_{\mathbb{H}^{n}}(x,y')  -
\nabla \mathcal{P}_{\mathbb{H}^{n}}(x,w'),l\rangle \big|\cdot
 \chi_{\Omega_{j}}(y')=0,
$$
which means for any $\varepsilon>0$, there exists a positive integer $m_{2}=m_{2}(\varepsilon,x,w)$ such that for any $j\geq m_{2} $ and $y'\in \mathbb{R}^{n-1}$,
$$
\big| \langle \nabla \mathcal{P}_{\mathbb{H}^{n}}(x,y')  -
\nabla \mathcal{P}_{\mathbb{H}^{n}}(x,w'),l\rangle \big|\cdot
 \chi_{\Omega_{j}}(y')<\varepsilon.
$$
Since
$\int_{\mathbb{R}^{n-1}}
\frac{ \chi_{\Omega_{j}}(y')}{||\chi_{\Omega_{j}}||_{L^{1}(\mathbb{H}^{n},\mathbb{R})}}
dV_{n-1}(y')=1$, then for any $j\geq m_{2} $,
\begin{eqnarray*}
&\;\;&\left| \int_{\mathbb{R}^{n-1}}
\langle \nabla \mathcal{P}_{\mathbb{H}^{n}}(x,y'),l\rangle
 \frac{ \chi_{\Omega_{j}}(y')}{ || \chi_{\Omega_{j}} ||_{L^{1}(\mathbb{H}^{n},\mathbb{R})}}
  dV_{n-1}(y')
 -  \langle \nabla \mathcal{P}_{\mathbb{H}^{n}}(x,w'),l\rangle \right|\\
&\leq&\int_{\mathbb{R}^{n-1}}
  \left|\langle \nabla \mathcal{P}_{\mathbb{H}^{n}}(x,y')  -
\nabla \mathcal{P}_{\mathbb{H}^{n}}(x,w'),l\rangle \cdot\chi_{\Omega_{j}}(y')\right|
 \cdot
 \frac{ \chi_{\Omega_{j}}(y')}{ || \chi_{\Omega_{j}} ||_{L^{1}(\mathbb{H}^{n},\mathbb{R})}}dV_{n-1}(y')
 \leq\varepsilon.
\end{eqnarray*}
Combining this with \eqref{eq-6.7}, we  conclude
\begin{eqnarray*}
\lim_{j\rightarrow\infty} \langle\nabla u_{j}(x),l\rangle
=\lim_{j\rightarrow\infty} \int_{\mathbb{R}^{n-1}}
 \langle\nabla \mathcal{P}_{\mathbb{H}^{n}}(x,y'),l\rangle
 \frac{ \chi_{\Omega_{j}}(y')}{ || \chi_{\Omega_{j}} ||_{L^{1}(\mathbb{H}^{n},\mathbb{R})}}
 dV_{n-1}(y')
 =\langle\nabla  \mathcal{P}_{\mathbb{H}^{n}}(x,w'),l\rangle.
\end{eqnarray*}
Replacing $x$ by $w$ and $l$ by $\pm e_{n}$ in the above equalities,
we obtain from \eqref{eq-6.6} that
$$
\lim_{j\rightarrow\infty} \langle\nabla u_{j}(w),\pm e_{n}\rangle
 = \langle\nabla \mathcal{P}_{\mathbb{H}^{n}}(w,w'),\pm e_{n}\rangle
=\frac{2^{n-2} (n-1)\Gamma(\frac{n}{2})}{ \pi^{\frac{n}{2}}x_{n}^{n}}\lim_{j\rightarrow\infty}
 ||\phi_{j} ||_{L^{1}(\mathbb{H}^{n},\mathbb{R})}.
$$
Hence, the sharpness of \eqref{eq-6.5} follows and we get
\begin{eqnarray} \label{eq-6.8}
\mathbf{C}_{\mathbb{H}^{n},\infty}(x)
=\sup_{l\in \mathbb{S}^{n-1}}\mathbf{C}_{\mathbb{H}^{n},\infty}(x;l)
=\mathbf{C}_{\mathbb{H}^{n},\infty}(x;\pm e_{n})
=\frac{2^{n-2} (n-1)\Gamma(\frac{n}{2})}{ \pi^{\frac{n}{2}}x_{n}^{n}}.
\end{eqnarray}
 The proof of Theorem \ref{thm-1.6}(2) is complete.

\smallskip
 (II) By \eqref{eq-1.4} and \eqref{eq-6.8},
we see that Theorem \ref{thm-1.7}(2) is true.
 \qed

\medskip

\subsection{Proofs of Theorems \ref{thm-1.1}$\sim$\ref{thm-1.5} for the half-space case}

For $l\in\mathbb{S}^{n-1}$, choose an unitary transformation $A$
 such that $Ae_{n}=l_{\alpha}$ and $Al=e_{n}$,
 where $l_{\alpha}=\sin\alpha \cdot e_{n-1}+\cos\alpha \cdot e_{n}$
and $\alpha \in[0,\pi]$.
It follows from Lemma \ref{lem-5.3} that
\begin{eqnarray}\label{eq-6.9}
\quad
C_{\mathbb{H}^{n},q}(x;l)
 &=&\frac{1}{2^{(n-1)q}} \int_{\mathbb{S}^{n-1}}|\langle A\eta,Al\rangle|^{q}\big(1+\langle Ae_{n},A\eta\rangle\big)^{(n-1)(q-1)} dS_{n-1}(A\eta)\\\nonumber
 &=&\frac{1}{2^{(n-1)q}} \int_{\mathbb{S}^{n-1}}|\langle\zeta,e_{n}\rangle|^{q}
 \big(1+\langle l_{\alpha},\zeta\rangle\big)^{(n-1)(q-1)} dS_{n-1}(\zeta)\\\nonumber
 &=&\frac{1}{2^{(n-1)(2q-1)}} \int_{\mathbb{S}^{n-1}}|\langle\zeta,e_{n}\rangle|^{q}
 | \zeta- l_{\alpha} |^{2(n-1)(q-1)} dS_{n-1}(\zeta).
\end{eqnarray}

(I) First, we prove Theorems \ref{thm-1.2}(2) and \ref{thm-1.5}(2).
Assume that $q\in[\frac{2K_{0}-1}{n-1}+1,\frac{2K_{0}}{n-1}+1]\cap [1,\infty)$, where $K_{0}\in \mathbb{N}$.
By \eqref{eq-2.8}, \eqref{eq-3.2} and \eqref{eq-3.4},
we know that for any  $\rho\in[0,1)$,
\begin{eqnarray}\label{eq-6.10}
&&\frac{\sqrt{\pi}\Gamma(\frac{q+n}{2})}{\Gamma(\frac{q+1}{2})\Gamma(\frac{n}{2}) (1+\rho^{2})^{(n-1)(q-1)}}
\int_{\mathbb{S}^{n-1}} |\langle\zeta,e_{n}\rangle|^{q} |\zeta-\rho l_{\alpha}|^{2(n-1)(q-1)}\; d\sigma(\zeta)\\\nonumber
&\leq&\;_{3}F_{2}\left(\frac{(n-1)(1-q)}{2},\frac{1-(n-1)(q-1)}{2},\frac{q+1}{2};\frac{1}{2},\frac{q+n}{2};
\frac{4\rho^{2}}{(1+\rho^{2})^{2}}\right)
\end{eqnarray}
and
\begin{eqnarray}\label{eq-6.11}
&&\frac{\sqrt{\pi}\Gamma(\frac{q+n}{2})}{\Gamma(\frac{q+1}{2})\Gamma(\frac{n}{2}) (1+\rho^{2})^{(n-1)(q-1)}}
\int_{\mathbb{S}^{n-1}} |\langle\zeta,e_{n}\rangle|^{q} |\zeta-\rho l_{\alpha}|^{2(n-1)(q-1)}\; d\sigma(\zeta)\\\nonumber
&\geq&\;_{2}F_{1}\left(\frac{(n-1)(1-q)}{2},\frac{1-(n-1)(q-1)}{2};\frac{q+n}{2};
\frac{4\rho^{2}}{(1+\rho^{2})^{2}}\right).
\end{eqnarray}
Observe that
$$
\frac{1}{2}+\frac{q+n}{2}-\frac{(n-1)(1-q)}{2}-\frac{1-(n-1)(q-1)}{2}-\frac{q+1}{2}>0
$$
and
$$
\frac{q+n}{2}-\frac{(n-1)(1-q)}{2}-\frac{1-(n-1)(q-1)}{2}>0.
$$
Then we deduce from \cite[Chapter 5]{rain} that the above two series are absolutely convergent when $\rho\leq1$.
Let $\rho=1$.
It follows from \eqref{eq-6.9}$\sim$\eqref{eq-6.11} that
\begin{eqnarray}\label{eq-6.12}
&&
C_{\mathbb{H}^{n},q}(x;l)\\\nonumber
&\leq&\frac{\Gamma(\frac{q+1}{2})\Gamma(\frac{n}{2})\omega_{n-1}}{\sqrt{\pi} 2^{q(n-1) }\Gamma(\frac{q+n}{2})}
\;_{3}F_{2}\left(\frac{(n-1)(1-q)}{2},\frac{1-(n-1)(q-1)}{2},\frac{q+1}{2};\frac{1}{2},\frac{q+n}{2};
1\right)
\end{eqnarray}
and
\begin{eqnarray} \label{eq-6.13}
&&
C_{\mathbb{H}^{n},q}(x;l)\\\nonumber
&\geq&\frac{\Gamma(\frac{q+1}{2})\Gamma(\frac{n}{2})\omega_{n-1}}{\sqrt{\pi}2^{q(n-1)}\Gamma(\frac{q+n}{2})}
\;_{2}F_{1}\left(\frac{(n-1)(1-q)}{2},\frac{1-(n-1)(q-1)}{2};\frac{q+n}{2};
1\right),
\end{eqnarray}
where $\omega_{n-1}$ denotes the $(n-1)$-dimensional Lebesgue measure on $\mathbb{S}^{n-1}$.
The equality holds in \eqref{eq-6.10} for $l_{\alpha}=\pm e_{n}$,
which means the equality holds in \eqref{eq-6.12} when $l=\pm  e_{n}$.
Further, the equality holds in \eqref{eq-6.11} for $l_{\alpha}=\pm e_{n-1}$,
which means the equality holds in \eqref{eq-6.13} when $l=t_{e_{n}}$, where $t_{e_{n}}\in \mathbb{S}^{n-1}$ and $\langle t_{e_{n}},e_{n}\rangle=0$.
These, together with Lemma \ref{lem-5.2} and \eqref{eq-5.8},
yield that
for any $x\in\mathbb{ H}^{n}$ and $l\in \mathbb{S}^{n-1}$,
\begin{eqnarray} \label{eq-6.14}
\mathbf{ C}_{\mathbb{H}^{n},q}( x; t_{e_{n}})
 \leq \mathbf{C}_{\mathbb{H}^{n},q}(x;l)
\leq \mathbf{C}_{\mathbb{H}^{n},q}(x; \pm e_{n})
=\mathbf{ C}_{\mathbb{H}^{n},q}( x).
\end{eqnarray}
Hence, Theorem \ref{thm-1.2}(2) holds true
and  Theorem \ref{thm-1.5}(2) follows from \eqref{eq-1.4}, \eqref{eq-5.8}, \eqref{eq-6.12} and Lemma \ref{lem-5.2}.
\smallskip

(II) Next, we show Theorems \ref{thm-1.1}(2) and \ref{thm-1.4}(2).
Assume that $q\in(1,\frac{n}{n-1})$.
By \eqref{eq-2.8}, \eqref{eq-3.6}, \eqref{eq-3.7} and \eqref{eq-6.9},
we arrive at the following two sharp estimates:
\begin{eqnarray*}
&&\;\;
C_{\mathbb{H}^{n},q}(x;l)\\\nonumber
&\geq&\frac{\Gamma(\frac{q+1}{2})\Gamma(\frac{n}{2})\omega_{n-1}}{\sqrt{\pi} 2^{q(n-1) }\Gamma(\frac{q+n}{2})}
\;_{3}F_{2}\left(\frac{(n-1)(1-q)}{2},\frac{1-(n-1)(q-1)}{2},\frac{q+1}{2};\frac{1}{2},\frac{q+n}{2};
1\right),
\end{eqnarray*}
and
\begin{eqnarray}\label{eq-6.15}
&&
C_{\mathbb{H}^{n},q}(x;l)\\\nonumber
&\leq&\frac{\Gamma(\frac{q+1}{2})\Gamma(\frac{n}{2})\omega_{n-1}}{\sqrt{\pi}2^{q(n-1)}\Gamma(\frac{q+n}{2})}
\;_{2}F_{1}\left(\frac{(n-1)(1-q)}{2},\frac{1-(n-1)(q-1)}{2};\frac{q+n}{2};
1\right).
\end{eqnarray}
Further, for any $x\in\mathbb{ H}^{n}$ and $l\in \mathbb{S}^{n-1}$, we obtain from \eqref{eq-2.8}, \eqref{eq-3.8}, \eqref{eq-6.9} and Lemma \ref{lem-5.2} that
\begin{eqnarray*}
\mathbf{C}_{\mathbb{H}^{n},q}(x; e_{n})
 \leq \mathbf{C}_{\mathbb{H}^{n},q}(x;l)
\leq\mathbf{ C}_{\mathbb{H}^{n},q}( x; t_{e_{n}})
=\mathbf{ C}_{\mathbb{H}^{n},q}( x).
\end{eqnarray*}
Therefore, Theorem \ref{thm-1.1}(2) holds true
and  Theorem \ref{thm-1.4}(2) follows from Lemma \ref{lem-5.2}, \eqref{eq-5.8} and \eqref{eq-6.15}.

\smallskip
(III) Now, we present the proof of  Theorem \ref{thm-1.4}(2).
For any $l\in \mathbb{S}^{n-1}$ and $x\in \mathbb{H}^{n}$,
by letting $q=\frac{n}{n-1}$ in \eqref{eq-6.12} and \eqref{eq-6.13},
we have
\begin{eqnarray}\label{eq-6.16}
&&
C_{\mathbb{H}^{n}, \frac{n}{n-1}}(x;l)
\equiv\frac{\pi^{\frac{n}{2}}\Gamma(\frac{2n-1}{2n-2})}{  2^{n-\frac{1}{2} }\Gamma(\frac{n^{2}}{2n-2})}
\quad\text{and}\quad
\mathbf{C}_{\mathbb{H}^{n}, \frac{n}{n-1}}(x)\equiv
\mathbf{C}_{\mathbb{H}^{n}, \frac{n}{n-1}}(x;l).
\end{eqnarray}
Similarly, for any $l\in \mathbb{S}^{n-1}$ and $x\in \mathbb{H}^{n}$,
by letting $q=1$ in \eqref{eq-6.12} and \eqref{eq-6.13}, we get
\begin{eqnarray}\label{eq-6.17}
&&
C_{\mathbb{H}^{n},1}(x;l)
\equiv
\frac{ 2^{2-n } \pi^{\frac{n-1}{2}}}{ \Gamma(\frac{1+n}{2})}
\quad\text{and}\quad
\mathbf{C}_{\mathbb{H}^{n},1}(x)\equiv
\mathbf{C}_{\mathbb{H}^{n}, 1}(x;l).
\end{eqnarray}
Therefore, Theorem \ref{thm-1.4}(2) follows from Lemma \ref{lem-5.2}, \eqref{eq-5.8}, \eqref{eq-6.16} and \eqref{eq-6.17}.
  \qed

\begin{remark}\label{rem-6.1}
In the case $q=2$, we can find a very explicit sharp point estimate for the half-space case.
For any $x\in \mathbb{H}^{n}$ and $l\in \mathbb{S}^{n-1}$,
it follows from \eqref{eq-5.8}, \eqref{eq-6.9}, \eqref{eq-6.14}, Lemma \ref{lem-5.2} and \cite[Lemma 1]{2019mele}
(or \cite[Theorem G]{chen2018}) that
\begin{eqnarray*}
C_{\mathbb{H}^{n},2}(x;l)
&\leq&C_{\mathbb{H}^{n},2}(x;\pm e_{n})=
\frac{1}{2^{2n-2 }} \int_{\mathbb{S}^{n-1}}|\langle\eta,e_{n}\rangle |^{2}
 |1-\langle\eta,e_{n}\rangle|^{2(n-1) } dS_{n-1}(\eta)\\
&=& \frac{\pi^{\frac{n-1}{2}}}{2^{2n-3}\Gamma(\frac{n-1}{2})} \int_{-1}^{1}(1-t^2)^{\frac{n-3}{2}}t^{2}
(1-t)^{n-1} dt.
\end{eqnarray*}
Let $t=\cos2s$, where $t\in[0,\frac{\pi}{2}]$.
Then
\begin{eqnarray*}
C_{\mathbb{H}^{n},2}(x;\pm e_{n})=
 \frac{\pi^{\frac{n-1}{2}}}{2^{n-3}\Gamma(\frac{n-1}{2})} \int_{0}^{\frac{\pi}{2}} (\sin2s)^{n-2} (\cos2s)^{2} (\sin s)^{2n-2}
ds.
\end{eqnarray*}
Note that for any $i,j\geq0$,
$$\int_{0}^{\frac{\pi}{2}} \sin^{i} s\cos^{j} s ds
=\frac{\Gamma(\frac{i+1}{2})\Gamma(\frac{j+1}{2})}{2\Gamma(\frac{i+j}{2}+1)}
$$
(cf. \cite[Page 19]{rain}).
Therefore,
\begin{eqnarray*}
 \int_{0}^{\frac{\pi}{2}} (\sin2s)^{n-2} (\cos2s)^{2} (\sin s)^{2n-2}
ds=
 \frac{2^{n-1}\Gamma(\frac{n+3}{2})\Gamma(\frac{3n-3}{2})}{ \Gamma(2n)}.
\end{eqnarray*}
Hence, we obtain from Lemma \ref{lem-5.2} and \eqref{eq-5.8} that
$$
\mathbf{ C}_{\mathbb{H}^{n},2}(x;l)
\leq \mathbf{C}_{\mathbb{H}^{n},2}(x)
 =\mathbf{C}_{\mathbb{H}^{n},2}(x;\pm e_{n})
 =
 \frac{2^{n-2} (n-1)\Gamma(\frac{n}{2})}{ \pi^{\frac{n}{2}}x_{n}^{\frac{n+1}{2}}}
\left(\frac{4\pi^{\frac{n-1}{2}} \Gamma(\frac{n+3}{2}) \Gamma(\frac{3n-3}{2}) }{  \Gamma(\frac{n-1}{2})\Gamma(2n )}\right)^{\frac{1}{2}}
$$
and
\begin{eqnarray*}
|\nabla u(x)|\leq
 \frac{2^{n-2} (n-1)\Gamma(\frac{n}{2})\|\phi\|_{L^{2}(\mathbb{R}^{n},\mathbb{R})}}{ \pi^{\frac{n}{2}}x_{n}^{\frac{n+1}{2}}}
\left(\frac{4\pi^{\frac{n-1}{2}} \Gamma(\frac{n+3}{2}) \Gamma(\frac{3n-3}{2}) }{  \Gamma(\frac{n-1}{2})\Gamma(2n )}\right)^{\frac{1}{2}} .
\end{eqnarray*}

\end{remark}

\vspace*{5mm}
\noindent{\bf Funding.}
 The first author is partially supported by NSFS of China (No. 11571216, 11671127 and 11801166),
 NSF of Hunan Province (No. 2018JJ3327), China Scholarship Council and the construct program of the key discipline in Hunan Province.
 The third author is partially supported by MPNTR grant 174017, Serbia.

\end{document}